\documentclass{amsart}
\usepackage{amsmath,amsfonts,amsthm,amssymb,tikz-cd,hyperref, mathrsfs}
\usepackage{cleveref}
\crefname{ineq}{Inequality}{inequalities}
\usepackage[shortlabels,inline]{enumitem}
\theoremstyle{theorem}
\newtheorem{theorem}{Theorem}[section]
\newtheorem{lemma}[theorem]{Lemma}
\newtheorem{proposition}[theorem]{Proposition}
\newtheorem{corollary}[theorem]{Corollary}

\theoremstyle{definition}
\newtheorem{example}[theorem]{Example}
\newtheorem{definition}[theorem]{Definition}

\newtheorem{remarklabeled}[theorem]{Remark}

\newtheorem{convention}[theorem]{Assumption}

\numberwithin{equation}{section}
\DeclareMathOperator{\GL}{GL}

\DeclareMathOperator{\SL}{SL}

\DeclareMathOperator{\Sp}{Sp}
\DeclareMathOperator{\tr}{tr}
\DeclareMathOperator{\id}{id}
\DeclareMathOperator{\fin}{fin}

\DeclareMathOperator{\Supp}{Supp}

\DeclareMathOperator{\Ad}{Ad}

\DeclareMathOperator{\reg}{reg}

\DeclareMathOperator{\Spec}{Spec}

\DeclareMathOperator{\Hom}{Hom}

\DeclareMathOperator{\rank}{rank}

\newcommand{\inverse}{^{-1}}
\newcommand{\gitquo}{/\!/}
\newcommand{\hamiltonianreduction}{/\!/\!/}

\newcommand{\M}{\mathcal{M}}
\newcommand{\g}{\mathfrak{g}}
\newcommand{\CC}{\mathbb{C}}
\newcommand{\ZZ}{\mathbb{Z}}
\newcommand{\PP}{\mathbb{P}}
\newcommand{\QQ}{\mathbb{Q}}

\newcommand{\mf}{\mathfrak}
\newcommand{\sheafO}{\mathcal{O}}
\newcommand{\cartanh}{\mathfrak{h}}

\newcommand{\mc}{\mathcal}

\newcommand{\gl}{\mathfrak{gl}}
\renewcommand{\sp}{\mathfrak{sp}}

\renewcommand{\sl}{\mathfrak{sl}}

\newcommand{\<}{\langle}
\renewcommand{\>}{\rangle}
\renewcommand{\L}{\mathcal{L}}
\title[Namikawa-Weyl groups of quiver varieties]{Namikawa-Weyl groups of affinizations of smooth Nakajima quiver varieties}
\author{Yaochen Wu}
\address{Department of Mathematics\\
	Yale University\\
New Haven, CT, 06511}
\email{yaochen.wu@yale.edu}

\begin{document}
	
	\begin{abstract}
		We give a description of the Namikawa-Weyl group of affinizations of smooth Nakajima quiver varieties based on combinatorial data of the underlying quiver, and compute some explicit examples. This extends a result of McGerty and Nevins for quiver varieties associated to Dynkin quivers. 
	\end{abstract}
\maketitle
	\setcounter{tocdepth}{1}
	
	\tableofcontents
	\section{Introduction}
	\subsection{Poisson deformations of conical symplectic singularity}
	We start by recalling the notion of conical symplectic singularity, first defined in \cite{beauville2000symplectic}, and their deformations.  
	Let $X$ be a normal affine Poisson variety. 
	\begin{definition}\label{conical symplectic singularities}
		We say $X$ \textit{has symplectic singularities} if its smooth locus $X^{\reg}$ is a symplectic variety with symplectic form $\omega$ and there is a projective resolution of singularities $\rho: Y \to X$ such that $\rho^* \omega$ extends to a regular 2-form on $Y$. 
		We say $X$ is \textit{conical} if there is a $\CC^*$-action on $X$ that contracts $X$ to a point, and $\omega$ has positive weight under this action. We say $\rho$ is a \textit{symplectic resolution of singularities} if $\rho^* \omega$ extends to a symplectic form on $Y$. 
	\end{definition}
	\begin{definition}\label{definition graded poisson deformation}
		Let $X$ be a normal affine Poisson variety equipped with a contracting $\CC^*$-action. A \textit{graded Poisson deformation} of $X$ is the data $(\mc{X}, B, j)$, where: 
		\begin{enumerate}
			\item $B = \bigoplus_{i\ge 0} B_i$ is a finitely generated positively graded $\CC$-algebra, such that $B_0 = \CC$. 
			\item $\mc{X}$ is an affine Poisson variety equipped with a $\CC^*$-action over $\Spec(B)$, and the structure morphism $\pi: \mc{X}\to \Spec(B)$ is $\CC^*$-equivariant and flat. 
			\item $j: X\xrightarrow{\sim} \pi\inverse(0)$ is a $\CC^*$-equivariant Poisson isomorphism, where $0\in \Spec(B)$ corresponds to the maximal ideal $\bigoplus_{i>0} B_i$. 
		\end{enumerate}
	\end{definition}
	Let $X$ be as in \Cref{definition graded poisson deformation} and $(\mc{X},B, j), (\mc{X}',B', j')$ be two graded Poisson deformations. A \textit{morphism} of graded Poisson deformations from $(\mc{X},B , j)$ to $ (\mc{X}',B', j')$ consists of $\CC^*$-equivariant morphisms $\Phi:\mc{X}\to \mc{X}'$ and  $f: \Spec(B)\to \Spec(B')$, such that the following diagram is Cartesian
	\[
	\begin{tikzcd}
		\mc{X}\arrow[r,"\Phi"] \arrow[d,"\pi"] & \mc{X}' \arrow[d,"\pi'"]\\
		\Spec(B)\arrow[r,"f"] &  \Spec(B')
	\end{tikzcd}
	\]
	and moreover, $j' = \Phi|_{\pi\inverse(0)} \circ j$. 
	
	We say a graded Poisson deformation $(\mc{X} ,B, j)$ is \textit{universal} if for any graded Poisson deformation $(\mc{X}',B', j')$, there is a unique morphism of graded Poisson deformations from $(\mc{X}',B', j')$ to $(\mc{X} ,B, j)$.

	Let $X$ be a conical symplectic variety, $\rho: Y\to X$ be a symplectic resolution of singularities. 
	\begin{theorem}{\label{Namikawa diagram of universal deformation }} {\cite[Theorem 5.5]{namikawa2011poisson}}
		There is a commutative diagram 
		\begin{equation}{\label{diagram of univ poisson defor of symplectic resol}}
			\begin{tikzcd}
				\mc{Y} \arrow[r] \arrow[d,"\pi_Y"] & \mc{X} \arrow[d,"\pi_X"]\\
				H^2(Y,\CC) \arrow[r,"q"]& \mathbb{A}^d
			\end{tikzcd}
		\end{equation}
		where $d= \dim H^2(Y,\CC)$, $\pi_X$, $\pi_Y$ are universal graded Poisson deformations of $X$ and $Y$ respectively, with $\pi_X\inverse(0) = X, \pi_Y\inverse(0) = Y$.

	\end{theorem}
	We write $\cartanh_X: = H^2(Y,\CC)$, which is known to depend only on $X$. Let $\mc{L}_1,...,\mc{L}_n$ be the codimension 2 symplectic leaves of $X$. 
	The formal slice $S_i$ to $\mc{L}_i$ is a Kleinian singularity, of type $A, D$ or $E$. Let $\widehat{W}_i$ be the corresponding Weyl group and $\hat{\cartanh}_i^*$ be the root space. The fundamental group $\pi_1(\mc{L}_i)$ acts on $\widehat{W}_i$ and $\hat{\cartanh}_i^*$ by Dynkin diagram automorphisms. 
	Define $W_i := (\widehat{W}_i)^{\pi_1(\mc{L}_i)}$ and $\cartanh_i := (\hat{\cartanh}_i^*)^{\pi_1(\mc{L}_i)}$. They are the Weyl groups and Cartan space corresponding to the folded Dynkin diagram by the automorphisms $\pi_1(\mc{L}_i)$. 
	
	\begin{definition}
		\begin{enumerate}
			\item The vector space $\cartanh_X$ is called the \textit{Namikawa-Cartan space} of $X$.
			\item The direct product
			\[W = \prod_i W_i,\]
			where $i$ runs over all the codimension 2 symplectic leaves of $X$, is called the \textit{Namikawa-Weyl group} of $X$. 
		\end{enumerate}
	\end{definition}
	
	The Namikawa-Weyl group is important for the following reason. 
	\begin{theorem}[{\cite[Theorem 1.1]{namikawa2010poisson}}]{\label{Namikawa Weyl group}}
		The map $q$ in \eqref{diagram of univ poisson defor of symplectic resol} is the quotient map of the action of $W$. 
	\end{theorem}
	
	The Namikawa-Cartan space $\cartanh_X$ has the following decomposition. 
	\begin{theorem}[{\cite[Lemma 2.8]{losev2022deformations}}]{\label{Ivan's decomposition of cartan space}} There is a vector space isomorphism $H^2(Y,\CC) = H^2(X^{\reg},\CC) \oplus \bigoplus_i \mf{h}_i$, where $i$ run through the codimension 2 leaves of $X$. 
	\end{theorem}

	\begin{remarklabeled}\label{remark how to project to cartanh i}
		Let us explain the projection $H^2(Y,\CC) \twoheadrightarrow \cartanh_i$, following \cite{namikawa2011poisson}. Let $\L_i$ be a codimension 2 symplectic leaf of $X$, and let $x\in \L_i$. Then, there is an analytic neighbourhood $U$ of $x$ in $X$, such that
		\begin{enumerate}
			\item there is a Poisson isomorphism $U \cong S_i \times \Delta^{\dim X -2}$, where $ \Delta^{\dim X -2}$ is the complex polydisc of dimension $\dim X -2$;
			\item $\rho\inverse(U) \cong \tilde{S}_i \times \Delta^{\dim X -2} $, where $\tilde{S}_i$ is the minimal resolution of the Kleinian singularity $S_i$. 
		\end{enumerate} 
		Take a class $\alpha\in H^2(Y,\CC)$, restrict it to $\rho\inverse(U)$, and we get a class $[\alpha]|_i \in H^2(\tilde{S_i},\CC)$. The latter is isomorphic to the root space $\hat{\cartanh}_i^*$. Namikawa showed (\cite[Proposition 4.2]{namikawa2011poisson}) $[\alpha]_i$ is invariant under the $\pi_1(\L_i)$ action, i.e. $[\alpha]_i \in \cartanh_i$. This is the projection of $[\alpha]$ to $\cartanh_i$. 
	\end{remarklabeled}

	\subsection{Quiver varieties}
	We recall the notion of quiver varieties, following \cite{nakajima1994instantons}. Let $Q$ be a quiver which may contain edge loops, $Q_0$ be the set of vertices and $Q_1$ the set of arrows. For $a\in Q_1$ let $t(a),h(a)$ denote the tail and head of $a$ respectively. 
	For $i,j\in Q_0$ let $n_{ij}$ denote the number of arrows between $i,j$, regardless of the orientation. 
	To each $i\in Q_0$ we associate a simple root $\alpha_i$. We identify $\ZZ_{\ge 0 }^{Q_0}$ with the root lattice and define the \textit{Tits form} on $\ZZ_{\ge 0}^{Q_0 }$ by 
	\[ \ZZ_{\ge 0}^{Q_0 } \times \ZZ_{\ge 0}^{Q_0} \to \ZZ, ((v_i)_i, (v'_i)_i) =v_i v'_i\sum_{i\in Q_0} (2- n_{ii} )- \sum_{i\neq j} v_iv'_j n_{ij}.\]
	
	Let $W_Q$ be the Weyl group of $Q$ generated by reflections on the root space $\QQ^{Q_0} = \text{Span}_\QQ \{\alpha_i|i\in Q_0\}$ along real simple roots, i.e. 
	\[s_i: \alpha \mapsto \alpha - (\alpha,\alpha_i) \alpha_i\]
	for a vertex $i$ that does not carry an edge loop. 
	
	For any quiver $Q$, let $\overline{Q}$ denote its \textit{double quiver}. By definition, it has the same vertex set as $Q$, and the set of arrows is $\overline{Q}_1 = \{a,a^*| a\in Q_1\}$, where $t(a^*) = h(a), h(a^*) = t(a)$. 
	
	For $v\in \ZZ_{\ge 0}^{Q_0 }$, the following definition will be frequently used: 
	\[p(v): = 1-\frac{1}{2}(v,v). \]
	
	Let $v,w\in\ZZ_{\ge 0}^{Q_0}$; let $V_i,W_i$ be vector spaces with $\dim V_i = v_i$ and $ \dim W_i = w_i$. Define the coframed representation space 
	\begin{equation}{\label{definition of R}}
		R(Q,v,w) := \bigoplus_{a\in Q_1} \Hom (V_{t(a)}, V_{h(a)}) \oplus \bigoplus_{i\in Q_0} \Hom (V_i,W_i).
	\end{equation}
	We omit $w$ when $w=0$ and write $R$ for $R(Q,v,w)$ when $Q,v,w$ is clear from the context. 
	The cotangent bundle $T^*R = R\oplus R^*$ carries a natural symplectic vector space structure; it can also be viewed as the representation space $R(\overline{Q},v,w)$ of the double quiver. 
	The group $G = \GL(v) := \prod_{i\in Q_0} \GL(V_i)$ acts on $R$ naturally; this induces a Hamiltonian $G$-action on $T^*R$. Let $\g := \prod_{i\in Q_0} \gl(v_i)$ be the Lie algebra of $G$. 
	We $G$-equivariantly identify $\g \cong \g^*$ via trace pairing. Then, the moment map $\mu: T^*R \to \g$ for the Hamiltonian $G$-action on $T^*R$ can be written as follows. 
	Let $(x,y,p,q) = (x_a,y_a,p_i,q_i)\in T^*R$ where \[x_a \in \Hom (V_{t(a)}, V_{h(a)}), y_a \in \Hom (V_{h(a)}, V_{t(a)}), q_i \in \Hom (V_i,W_i), p_i\in \Hom(W_i,V_i),\]
	$a\in Q_1, i\in Q_0$.  
	Then  
	\[\mu((x_a,y_a,p_i,q_i)) = [x,y]-pq.\] 
	More precisely, the $i$-component of the right hand side is 
	\[\sum_{h(a) = i} x_a y_a - \sum_{t(a) = i} y_a x_a -p_iq_i .\]
	
	Let $\theta$ be a character of $G$; it has the form $\theta (g) = \prod_{i\in Q_0} \det(g_i)^{\theta_i}$, $g_i \in \GL(v_i)$, $\theta_i \in\ZZ$. We can thus identify $\ZZ^{Q_0}$ and the character lattice of $G$, and write $\theta \in \ZZ^{Q_0}$. We write $\theta > 0$ if $\theta_i > 0$ for all $i\in Q_0$. Let $ (T^*R)^{\theta-ss}$ denote the set of $\theta$-semistable points of $T^*R$. If $\theta_i>0$ for all $i$ then $(x,y,p,q)$ being semistable is equivalent to $\ker q$ having no nonzero subspace stable under $x,y$. 
	
	We write $\mathfrak{p}:= (\g/[\g,\g])^*$, which is identified with $\CC^{Q_0}$. 
	Let $\lambda \in\mathfrak{p}$. 
	The GIT quotient 
	\[\M_\lambda^\theta (Q, v,w) : = (\mu\inverse(\lambda)^{\theta-ss})\gitquo G\] 
	is called the Nakajima quiver variety. It inherits a Poisson variety structure from the symplectic structure of $T^*R$. We omit the letter $Q$ when it is clear from context, and omit the letter $w$ when $w=0$.
	The variety $\M_\lambda^0 (v,w)$ is affine, and we have a projective morphism $\rho: \M_\lambda^\theta (v,w) \to \M_\lambda^0 (v,w)$. 
	
	For $\theta\in\ZZ^{Q_0}, \lambda\in \mathfrak{p}$, the pair $(\theta,\lambda)$ is said to be \textit{generic} if there are no positive roots $v'<v$ of $Q$ such that $\lambda\cdot v'=\theta\cdot v'= 0$. 
	We say $\lambda$ is generic if $(0,\lambda)$ is generic and we say $\theta$ is generic if $(\theta,0)$ is generic.
	When $(\theta,\lambda)$ is generic, $G$ acts freely on $(T^*R)^{\theta-ss}$ (\cite[Section 3.ii]{nakajima1998quiver}), and $\M_\lambda^\theta (v,w)$ is smooth. 
	
	We can also define families of Nakajima quiver varieties, i.e. 
	\[\M_\mathfrak{p}^\theta ( v,w) : = (\mu\inverse(\mathfrak{p})^{\theta-ss})\gitquo G.\]
	It is a scheme over $\mathfrak{p}$, and each $\M_\lambda^\theta (v,w)$ is the fiber over $\lambda\in \mathfrak{p}$. 
	
	Let $\Lambda_{w} = \sum w_i \varpi_i$ where $\varpi_i$ is the fundamental weight corresponding to the simple root $\alpha_i$. We will frequently use the weight $\nu$ defined by 
	\[\nu = \Lambda_w - \sum_{i\in Q_0} v_i \alpha_i.\]

	\begin{convention}
		Throughout the paper, we shall assume $v_i \neq 0$ for all $i\in Q_0$, unless otherwise specified. In fact, if $v_i = 0$ for some $i$, then we can view $\M_\lambda^\theta(Q,v,w)$ as $\M_\lambda^\theta(Q',v,w')$ where $Q'$ is the subquiver of $Q$ obtained by deleting $i$ and all arrows adjacent to it, and $w'$ is the restriction of $w$ to $\ZZ_{\ge 0}^{Q_0 \backslash \{i\}}$. 
	\end{convention}
	\begin{lemma}[{\cite[ Section 2.1.8]{bezrukavnikov2021etingof}}]
		For generic $\theta$, $\CC[\M_\lambda^\theta (v,w)]$ is a finitely generated algebra independent of $\theta$.
	\end{lemma}
	The above lemma enables us to make the following definition. 
	
	\begin{definition}
		For $\lambda\in\mathfrak{p}$, we set $\M_{\lambda} (v,w):= \Spec (\CC[\M_\lambda^\theta (v,w)])$ for generic $\theta$. We say $\M_{\lambda} (v,w)$ is the \textit{affinization} of $\M_\lambda^\theta (v,w)$. 
	\end{definition}
	The proof of {\cite[Proposition 2.3]{bezrukavnikov2021etingof}} generalizes to show that, for generic $\theta$, the natural map 
	$p: \M_\lambda^\theta (v,w) \to \M_\lambda (v,w)$ is a symplectic resolution of singularities. Therefore the variety $\M_0(v,w)$ is a conical symplectic singularity, \Cref{conical symplectic singularities}. The following observation will play an important role. 
	
	\begin{example}{\label{Mp is deformation over M0}}
		The morphism
		\[\M_\mf{p}(v,w):= \Spec\CC[\M_{\mf{p}}^\theta(v,w)] \to \mf{p}\] gives a graded Poisson deformation of $\M_0(v,w)$. 
	\end{example}
	\subsection{Namikawa-Weyl group of quiver varieties}
	
	The goal of this paper is to describe the Namikawa-Weyl group of quiver varieties $\M_0(v,w)$. Let us first record some known special cases. 
	Assume $Q$ is a simply-laced Dynkin quiver and $\nu = \Lambda_w-\sum v_i\alpha_i $ is a dominant weight. In this case it is known that $\M_0(v,w) \cong \M_0^0(v,w)$, see \Cref{for dominant nu and finite or affine quiver affinization = affine}. 
	McGerty and Nevins in \cite{mcgerty2019springer} described the Namikawa-Weyl group of $\M_0^0 (v,w)$ as follows. 
	Let $\Phi$ be the set of roots, $\Phi_\nu =\{\alpha \in \Phi|\<\nu,\alpha\> = 0\} $ and $\Phi_\nu^{\max}$ be the maximal elements of $\Phi_\nu$ with respect to the usual partial order on $\Phi$. 
	
	\begin{proposition}[{\cite[Lemma 5.1]{mcgerty2019springer}}]
		The set of codimension 2 symplectic leaves of $ \M_0^0 (v,w) $ is in bijection with $\Phi_\nu^{\max}$. 
	\end{proposition}
	
	\begin{proposition}[{\cite[Theorem 5.4]{mcgerty2019springer}}]{\label{MN weyl group dynkin quiver}}
		The Namikawa-Weyl group of $ \M_0^0 (v,w) $ is $W_\nu$, where $W_\nu$ is the Weyl group of the sub-root system $\Phi_\nu$ of $\Phi$.
	\end{proposition}
	The group $W_\nu$ coincides with the stabilizer of $\nu$ in $W_Q$. 
	
	The following theorem summarizes the main results of this paper.
	\begin{theorem}\label{summary in introduction}
		The following are true.
		\begin{enumerate}
			\item Weyl groups of type $A_n,B_n,D_n,E_6,E_7,E_8,G_2$ can appear as components of the Namikawa-Weyl group of the affinization of some quiver variety $\M_0(v,w)$. 
			\item Weyl groups of type $C_n,n\ge 3$ and $F_4$ cannot appear as components of the Namikawa-Weyl group of the affinization of any quiver variety $\M_0(v,w)$.
			\item Weyl groups of type $B_n, G_2$ can only appear when the underlying quiver is wild (i.e. not finite type or affine type). 
		\end{enumerate}
	\end{theorem}
	Part (1) follows from \Cref{MN weyl group dynkin quiver}, \Cref{Bk example} and \Cref{Nakajima G2 quiver via isotropic decomposition}. Part (2) is proved in \Cref{F4 cannot appear} and \Cref{Cn cannot appear}. Part (3) is proved in \Cref{multiply laced only for wild quiver}. 
	
	In addition, we give a description of the Namikawa-Weyl group in \Cref{extend mn result to affine type} when the underline quiver is affine type, which turns out to be similar to \Cref{MN weyl group dynkin quiver}. 
	\subsection{Structure of the paper}\label{structure of the paper}
	In \Cref{Section Notations and preliminaries} we present more results on quiver varieties. 
	In \Cref{section Tautological line bundles and the Kirwan}, we introduce tautological line bundles on quiver varieties, and use them to classify all the possible components of the Namikawa-Weyl group of the affinization of a quiver variety. 
	To carry out the methods in \Cref{section Tautological line bundles and the Kirwan}, we must know certain information of codimension 2 leaves of the quiver varieties (more precisely, \Cref{BS isotropic decomposition}), and this is usually difficult. Therefore, we take a different approach. 
	In \Cref{Codimension 2 leaves in deformed quiver varieties}, we examine the existence of codimension 2 symplectic leaves of deformations of $\M_0(v,w)$ along a subgeneric parameter, and fully classify them in \Cref{theorem cod 2 leaves in deformed quiver varietyes}. 
	In \Cref{Codimension 2 leaves in deformed conical symplectic singularity}, we recall general results about the presence of codimension 2 symplectic leaves in deformations of conical symplectic singularities. 
	In \Cref{Application to quiver varieties}, we use the results from the previous sections to compute the Namikawa-Weyl groups of some quiver varieties. 
	
	\subsection*{Acknowledgment} I am deeply grateful to Ivan Losev, without whom this paper would never appear, for suggesting this problem and for many fruitful discussions as well as numerous suggestions to improve the exposition. I would like to thank Travis Schedler for useful discussions, especially on relations between $\M_0$ and $\M_0^0$, and Gwyn Bellamy for pointing me to \cite{bellamy2020birational} which contains a result on Namikawa-Weyl groups of quiver varieties from affine type quivers. 
	I would like to thank Hiraku Nakajima for suggesting \Cref{Nakajima G2 quiver via isotropic decomposition}. 
	I would like to thank Do Kien Hoang and Dmytro Matvieievskyi for inspiring conversations. I would like to thank the anonymous referee for providing many suggestions that help improve the exposition and pointing out several mistakes in the earlier versions of this paper. 

	\section{More on quiver varieties}{\label{Section Notations and preliminaries}}
	In this subsection we present more results on quiver varieties. An important goal is to reduce our problem to \Cref{best convention on tilde v}, which largely simplifies subsequent computations. 
	\subsection{Framed v.s. non-framed quivers}{\label{equivalent description}}
	The following notations are given in \cite[Section 1, Remarks]{crawley2001geometry}. 
	Given a quiver $Q$, let ${Q^\infty}$ be the \textit{extended quiver} with ${Q^\infty_0} = Q_0 \cup \{\infty\}$; the arrows between vertices in $Q_0$ are the same as those in $Q_1$, plus $w_i$ arrows from the vertex $i$ to the vertex $\infty$. We denote the new simple root associated to the vertex $\infty$ by $\alpha_\infty$. 
	
	Given $Q,v,w$, we define the extended dimension vector $\tilde{v}\in \ZZ_{\ge 0}^{{Q^\infty_0}}$ by $\tilde{v}_\infty = 1$ and $\tilde{v}_i = v_i$ for $i\in Q_0$. 
	It is clear that $R(Q,v,w) = R({Q^\infty},\tilde{v})$ as in \eqref{definition of R}, and  $\M^\theta_\lambda(Q,v,w) = \M_{\tilde{\lambda}}^{\tilde{\theta}}({Q^\infty},\tilde{v})$, 
	where $\tilde{\lambda}_\infty = -v\cdot \lambda$, $\tilde{\lambda}_i = \lambda_i$ for $i\in Q_0$; $\tilde{\theta}_\infty = -\theta \cdot v, \tilde{\theta}_i = \theta_i$ for $i \in Q_0$. 
	
	By definition, $\<\nu,v'\> = -(\tilde{v},v')$ for all $v'\in \ZZ^{{Q_0}}$. We will use these two notions interchangeably. 
	
	\subsection{Stratification of affine quiver varieties}
	The quiver variety $\M_\lambda^0(v,w)$, which is isomorphic to $ \M_{\tilde{\lambda}}^0(\tilde{v})$, has a stratification by symplectic leaves, which we describe below.
	Let $x\in\M_\lambda^0(v,w)$ and  $r \in T^*R$ be a representative of $x$ (with closed $G$-orbit). 
	\begin{definition}
		Suppose $r =r_0 \oplus r_1^{\oplus n_1} \oplus ... \oplus r_k^{\oplus n_k}$, 
		where $r_i$'s are pairwise non-isomorphic simple representations of $\overline{Q^\infty}$, and $\dim(r_0)_\infty = 1$. Write $v^i = \dim r_i \in \ZZ_{\ge 0}^{{Q^\infty_0}}$. We say 
		\[\tau = (v^0,1;v^1,n_1;v^2,n_2;...;v^k,n_k)\]
		is the \textit{representation type} of $x$. 
	\end{definition}
	
	\begin{proposition}[{\cite[Section 3.v]{nakajima1998quiver}}]
		The Poisson variety $\M_\lambda^0(v,w)$ has finitely many symplectic leaves. The symplectic leaf containing $x$ consists of all elements that has the same representation type as $x$. 
	\end{proposition}
	Let us discuss what dimension vector $v^i$ can appear in a representation type $\tau$. 
	For $\tilde{\lambda}\in \CC^{{Q^\infty_0}}$, define the \textit{deformed preprojective algebra}
	\[\Pi^{\tilde{\lambda}} = \CC \overline{{Q^\infty}} /(\sum_{a\in {Q^\infty_1}} [a,a^*]- \sum_{i\in {Q^\infty_0}}\tilde{\lambda}_i e_i) \]
	where $\CC \overline{{Q^\infty}}$ is the path algebra of the double quiver of $ {{Q^\infty}}$, and $a^*$ is the opposite arrow to $a$ in the double quiver. For details see \cite[Section 2]{crawley1998noncommutative}. An element in $\mu\inverse(\lambda) \subset T^*R$ is the same as a representation of $\Pi^{\tilde{\lambda}}$. 
	We will frequently use the following results.
	
	\begin{theorem}[{\cite[Theorem 1.2]{crawley2001geometry}}]{\label{CB simple dimension criteria}}
		For $v\in \ZZ_{\ge 0}^{Q_0^\infty}$, the following conditions are equivalent.
		\begin{enumerate}
			\item There is a simple representation of $\Pi^{\tilde{\lambda}}$ (i.e. an element in  $\mu\inverse(\tilde{{\lambda}})$) with dimension vector $v$.
			\item $v$ is a positive root of ${Q^\infty}$, $\tilde{{\lambda}}\cdot v = 0$, and for any decomposition $v = \beta^1+ ... + \beta^n$ where $n\ge 2$ and $\beta^i$ are positive roots of ${Q^\infty}$ such that $\beta^i\cdot \tilde{{\lambda}}=0$, we have $p(v) > \sum_{i=1}^n p(\beta^i)$. 
		\end{enumerate}
	\end{theorem}
	We write $\Sigma_{\tilde{\lambda}}$ for the set of $v$ such that the conditions in \Cref{CB simple dimension criteria} hold. 
	
	\begin{theorem}[{\cite[Theorem 1.3]{crawley2001geometry}}]{\label{CB strata dimension formula}}
		Let $\tau = (v^0,n_0 = 1;v^1,n_1;...;v^k,n_k)$ be a representation type such that $\sum_{i=0}^k n_iv^i = \tilde{v}$.
		The stratum of $\M_{\tilde{{\lambda}}}^0(\tilde{v})$ associated to $\tau$ has dimension
		\[d(\tau) = 2\sum_{i=0}^k p(v^i).\]
	\end{theorem}
	
	\begin{lemma}{\label{Nonempty smooth quiver variety implies root}}
		If $\M_0^\theta(v,w)$ is nonempty for generic $\theta$, then $\tilde{v}$ is a root of $Q^\infty$. In particular, $w\neq 0$. 
	\end{lemma}
	
	\begin{proof}
		When $\lambda$ is generic, the $G$-action on $\mu\inverse(\lambda)$ is free, so $\mu\inverse(\lambda) = \mu\inverse(\lambda)^{\theta-ss}$ and $\M_\lambda^\theta(v,w) \cong \M_\lambda^0 (v,w)$. It is nonempty only if $\tilde{v}$ is a root due to \Cref{CB simple dimension criteria}. 
		Since $\mu$ is flat when restricted to the $\theta$-stable locus of $T^*R$, $\M_0^\theta(v,w)$ is nonempty only if $\M_\lambda^\theta(v,w)$ is nonempty. Hence $\tilde{v}$ is a root. If $w=0$, then $\tilde{v}$ has disconnected support and cannot be a root. 
	\end{proof}
	
	The following useful lemma is immediate from the definition of $p$.
	\begin{lemma}\label{When is p(a+b)>p(a)+p(b)}
		If $a,b\in\ZZ_{\ge 0}^{Q_0}$, then $p(a+b)>p(a)+p(b)$ if and only if $-(a,b) > 1$. 
	\end{lemma}
	
	\subsection{Local structure}{\label{Structure of neighbourhood}}
	In this subsection we follow \cite[section 2.1.6]{bezrukavnikov2021etingof}. Similar results under the hyper-Kahler setting are first given in \cite[section 6]{nakajima1994instantons}.

	For a fixed representation type $\tau = (v^0,1;v^1,n_1;v^2,n_2;...;v^k,n_k)$, we define a new quiver $\underline{Q}$. It has $k$ vertices, in bijection with the $k$ dimension vectors $\{v^1,...,v^k\}$. The number of arrows from $i$ to $j$ is $-(v^i,v^j)$ if $i\neq j$, and $1-\frac{1}{2}(v^i,v^i) = p(v^i)$ edge loops at each $i \ge 1$. 
	
	Let the dimension vector $\underline{v} \in \ZZ_{\ge 0}^{\underline{Q}_0}$ be defined by $\underline{v}_i = n_i$. Finally, let the framing $\underline{w} \in \ZZ_{\ge 0}^{\underline{Q}_0}$ be defined by $\underline{w}_i = -(v^0,v^i)$. The property of this construction is that the representation space $T^*R(\underline{Q},\underline{v}, \underline{w})$ satisfies 
	\begin{eqnarray}\label{iso of slice H modules}
		T^*(\g.r) \oplus T^*R(\underline{Q},\underline{v}, \underline{w}) \oplus \CC^{2-(v^0,v^0)} \cong T^*R .
	\end{eqnarray}
	Let $H = \prod_{i=1}^k \GL(n_i) $, so that $H$ is identified with the subgroup of $G$ of automorphisms of a representation with type $\tau$. 
	The isomorphism \eqref{iso of slice H modules} is an isomorphism of symplectic $H$-modules, with trivial $H$-action on $\CC^{2-(v^0,v^0)}$. 
	Write $R_0:= \CC^{2-(v^0,v^0)}$ and $V := T^*R(\underline{Q},\underline{v}, \underline{w}) \oplus R_0$. 
	
	Let $\underline{\mathfrak{p}}:= \CC^k \cong (\cartanh/[\cartanh,\cartanh])^*$, where $\cartanh$ is the Lie algebra of $H$, viewed as a Lie subalgebra of $\g$.  For $\theta\in \ZZ^{Q_0}$, define
	\[\underline{\M}_{\underline{\mathfrak{p}}}^{\theta}(\underline{v},\underline{w}) := \M_{\underline{\mathfrak{p}}}^{\theta}(\underline{Q},\underline{v},\underline{w})\] 
	Here we abuse notation and write $\theta$ for the restriction of $\theta$ to $H$. 
	We define $\underline{\M}_{{\mathfrak{p}}}^{\theta}(\underline{v},\underline{w}) := \mathfrak{p} \times_{\underline{\mathfrak{p}}}\underline{\M}_{\underline{\mathfrak{p}}}^{\theta}(\underline{v},\underline{w})$, 
	via the shifted restriction morphism $\mathfrak{p}\to \underline{\mathfrak{p}}, \eta \mapsto (\eta-\lambda)|_\cartanh$.
	
	Define  $\M_{\mathfrak{p}}^0(v,w)^{\wedge_x} :=\Spec (\CC[\M_\mathfrak{p}^0(v,w)]^{\wedge_x})$ where $\wedge_{x}$ in the right hand side denotes the completion with respect to the maximal ideal of $x$. 
	Define $(\underline{\M}_\mathfrak{p}^0(\underline{v},\underline{w})\times R_0)^{\wedge_0}$ similarly. 
	For generic $\theta$, define 
	\[\M_\mathfrak{p}^\theta(v,w)^{\wedge_x} := \M_\mathfrak{p}^0(v,w)^{\wedge_x} \times_{\M_\mathfrak{p}^0(v,w)} \M_\mathfrak{p}^\theta (v,w);\] and define
	$\underline{\M}^{\theta}_\mathfrak{p}(\underline{v},\underline{w})^{\wedge_0}$ similarly. The following theorem describes the local structure of $x$ in $M_\mathfrak{p}^\theta (v,w)$.
	
	\begin{theorem}[{\cite[Section 2.1.6]{bezrukavnikov2021etingof}}]{\label{structure of nbhd theorem}}
		There is a commutative diagram 
		\[\begin{tikzcd}
			\M_\mathfrak{p}^\theta (v,w)^{\wedge_x}  \arrow[d,"\rho"] \arrow[r,"\sim"] 
			&
			(\underline{\M}^\theta_\mathfrak{p}(\underline{v},\underline{w})\times R_0)^{\wedge_0} \arrow[d,"\underline{\rho} \times \id"] \\
			\M_\mathfrak{p}^0(v,w)^{\wedge_x}
			\arrow[r,"\sim"] & (\underline{\M}_\mathfrak{p}^0(\underline{v},\underline{w}) \times R_0)^{\wedge_0}
		\end{tikzcd}, \]
		where the horizontal morphisms are isomorphisms, and $\hat\rho: \underline{\M}^\theta_0(\underline{v},\underline{w})\to \underline{\M}_0^0(\underline{v},\underline{w})$ is the natural projective map. 
	\end{theorem}
	Specializing to $\lambda\in\mathfrak{p}$, we get isomorphisms $\M_\lambda^\theta (v,w)^{\wedge_x} \cong (\underline{\M}^\theta_0(\underline{v},\underline{w})\times R_0)^{\wedge_0}$, where $\theta$ is generic or $0$. 
	
	Let us record an analytic version of the local structure theorem. Consider the homogeneous bundle $G\times^H ((\g/\cartanh)^*\oplus V)$. It is isomorphic to the Hamiltonian reduction $(T^*G\times V)\hamiltonianreduction_0H$, and is therefore symplectic. The natural $G$ action on $G\times^H ((\g/\cartanh)^*\oplus V)$ is Hamiltonian, and the moment map is given by $\mu_G([g,\alpha,v]) = \Ad g(\alpha+ \mu_H(v))$, where $\mu_H$ is the moment map for the $H$-action on $V$.
	
	Let $\pi: T^*R \to T^*R\gitquo G$, and $\underline{\pi}: G\times^H ((\g/\cartanh)^*\oplus V) \to  ((\g/\cartanh)^*\oplus V) \gitquo H$ be the categorical quotients by $G$. The following proposition is an application of \cite[Proposition 3]{losev2006symplectic}.
	\begin{proposition}\label{analytic local structure theorem}
		There is an analytic neighbourhood $U$ of $x$ in $\M_\mathfrak{p}^0(v,w)$, an analytic neighbourhood $\underline{U}$ of $0$ in $\underline{\M}_\mathfrak{p}^0(\underline{v},\underline{w}) \times R_0$, and a symplectic isomorphism $\phi$ of analytic $U$ and $\underline{U}$ intertwining the morphisms 
		$\M_\mathfrak{p}^0(v,w) \to \mathfrak{p}, \underline{\M}_\mathfrak{p}^0(\underline{v},\underline{w})\to \mathfrak{p}$. It lifts to an isomorphism 
		\[\tilde{\phi}: \pi\inverse(U) \to \underline{\pi}\inverse (\underline{U})\]
		of $G$-stable analytic neighbourhoods of the orbits $Gr\subset T^*R$ and of $G.[1,0,0]\subset G\times^H ((\g/\cartanh)^*\oplus V) $; $\tilde{\phi}$ is symplectic, $G$-equivariant, and intertwines moment maps. 
	\end{proposition}

	\subsection{Maffei's isomorphism}{\label{Maffei's isomorphism}}
	Suppose $(\theta,\lambda)$ is generic. Then for any $\sigma\in W_Q$, the Weyl group of the quiver $Q$, we have an isomorphism 
	\[\M_\lambda^\theta(v,w) \cong \M_{\sigma\lambda}^{\sigma\theta}(\sigma\bullet v,w)\]
	where we view $\sigma$ as an element of $W_{Q^\infty}$, $\sigma\bullet v = \sigma(v+\alpha_\infty)- \alpha_\infty$ (see \Cref{equivalent description}), $\sigma \theta$ and $\sigma \lambda$ are defined so that $(\sigma \theta)\cdot (\sigma v) = \theta \cdot v$, $(\sigma \lambda)\cdot (\sigma v) = \lambda \cdot v$. Equivalently, the weight $\nu$ corresponding to $(\sigma\bullet v,w)$ is $\sigma\nu$, and $\sigma$ acts on $\ZZ^{Q^\infty_0}$ by $\tilde{v}\mapsto \sigma\tilde{v}$. 
	These isomorphisms are first proved by Maffei in \cite{maffei2002remark}. We refer to them as Maffei's isomorphisms. 
	\begin{remarklabeled}{\label{use maffei to assume nu dominant}}
		We can find $\sigma \in W_{Q}$ 
		such that $\sigma\tilde{v}$ has minimal height with respect to the simple real roots of $Q$; equivalently, $(\sigma\tilde{v}, \alpha_i)\le 0$ for all $i\in Q_0$ (this is automatically true for imaginary simple roots); equivalently, $\nu' = \Lambda_{w} - \sum_{i\in Q_0} (v')_i\alpha_i$ is dominant. 
		Therefore, we may always assume $\nu$ is dominant; equivalently, $(\tilde{v},\alpha_i)\le 0$ for all $i\in Q_0$. Note that this does not imply $(\tilde{v},\alpha_\infty) \le 0$. 
	\end{remarklabeled}

	\subsection{$\M_\lambda(v,w)$ v.s. $\M_\lambda^0(v,w)$}
	In this subsection we study the relation between affine quiver varieties and affinizations of smooth quiver varieties. 
	
	\begin{proposition}[{\cite[Corollary 2.4, Proposition 2.5]{bezrukavnikov2021etingof}}]
		\label{flat mu implies affinization equals affine}
		If the moment map $\mu$ is flat, then $\M_\lambda^0(v,w) \cong \M_\lambda(v,w)$. 
	\end{proposition}
	Let us record some criteria for flatness of $\mu$. 
	\begin{theorem}[{\cite[Theorem 1.1]{crawley2001geometry}}]\label{CB flatness of moment map theorem}
		Fix a quiver $Q$ and a dimension vector $v$ (without framing). The following are equivalent. 
		\begin{enumerate}
			\item $\mu: T^*R\to \g$ is flat.
			\item $\dim \mu\inverse(0) = v\cdot v - 1+ 2p(v) = \dim T^*R - \dim G +1$. 
			\item For any decomposition $v = \beta_1+ \cdots + \beta_k$, where all $\beta_i$ are positive roots, $p(v) \ge \sum_{i=1}^k p(\beta_i)$. 
			\item For any decomposition $v = \beta_1+ \cdots + \beta_k$, where all $\beta_i \in \ZZ_{\ge 0}^{Q_0}$, $p(v) \ge \sum_{i=1}^k p(\beta_i)$. 
		\end{enumerate}
	\end{theorem}
	In particular, if $\tilde{v}\in\Sigma_0$, then $\mu$ is flat. Moreover, it is clear by \Cref{CB simple dimension criteria} that $\tilde{v}\in\Sigma_0$ implies $\tilde{v}\in\Sigma_{\tilde{\lambda}}$ for all $\lambda\in \mathfrak{p}$. Thus, we get the following corollary of \Cref{flat mu implies affinization equals affine}. 
	\begin{corollary}
		If $\tilde{v}\in\Sigma_0 $, then $\M_\lambda^0(v,w) \cong \M_\lambda(v,w)$ for all $\lambda\in \mathfrak{p}$.
	\end{corollary}
	
	\begin{corollary}{\label{for dominant nu and finite or affine quiver affinization = affine}}
		If $Q$ is a finite or affine type quiver, and $\nu = \Lambda_w - \sum_{i\in Q_0} v_i \alpha_i$ is dominant, then $\M_0^0(v,w) \cong \M_0(v,w)$. 
	\end{corollary}
	\begin{proof}
		By \cite[Lemma 2.1]{bezrukavnikov2021etingof}, $\mu$ is flat in these cases. 
	\end{proof}
	
	Although it is not always true that $\M_0^0(v,w) \cong \M_0(v,w)$, we have the following result. 
	
	\begin{proposition}\label{affinization is an aff quiv variety by Maffei}
		Suppose $\M_0(Q,v,w)$ is nonempty. Then \[\M_0(Q,v,w)\cong \M_0^0(Q',v',w')\]for possibly different $Q',v'$ and $w'$. 
	\end{proposition}
	\begin{proof}
		By \Cref{use maffei to assume nu dominant} we may assume $(\tilde{v},\alpha_i) \le 0$ for all simple roots $\alpha_i$. The assumption that $\M_0(Q,v,w)$ is nonempty implies $\tilde{v}$ is a root, so $w\neq 0$. 
		
		If $(\tilde{v},\alpha_\infty) \le 0$ as well, then $\tilde{v}$ lie in the fundamental domain $ F_0$. If $\tilde{v}\in \Sigma_0$ then $\mu$ is flat. Otherwise, $\tilde{v}$ falls in the 3 cases of \cite[Theorem 8.1]{crawley2001geometry}. Case (I) there cannot appear since the multiplicity of $\alpha_\infty$ is 1 in $\tilde{v}$, and in Case (II) and (III), $\mu$ is flat by \cite[Theorem 1.1]{su2006flatness}. 
		Therefore, $\M_0^0(Q,v,w) \cong \M_0(Q,v,w)$ by \Cref{flat mu implies affinization equals affine}.
		
		Suppose now $(\tilde{v},\alpha_\infty) = 2+(v,\alpha_\infty) >0$. Then since $w\neq 0$, we must have $(v,\alpha_\infty)  = -1$. Therefore, the vertex $\infty$ is connected to a unique vertex $i \in Q_0$, and $v_i = 1$.
		
		Therefore, we can view the vertex $i$ are a framing and $\infty$ as a usual vertex, and apply Maffei's isomorphism for $s_\infty$. Then we see $\M_0(Q,v,w) \cong \M_0(Q',v',w')$, where $Q'$ is obtained from $Q$ by deleting the vertex $i$, $v'_j = v_j$ for $j\in Q'_0$ and the new framings $w'_j$ equals the number of arrows between $i$ and $j$. This operation cuts down the size of the underlying quiver. Repeat this process if necessary, and we either get a $\M_0(Q',v',w')$ for which $\tilde{v'}\in F_0$, or get $Q'$ the quiver where all vertices has edge loops. In either case, the moment map $\mu$ for $(Q',v',w')$ is flat by \cite[Theorem 1.1]{su2006flatness}, and we have $\M_0(Q,v,w) \cong \M_0(Q',v',w')\cong  \M_0^0(Q',v',w')$. 
	\end{proof}
	Thanks to the above proof, we may always make the following assumption.
	\begin{convention}\label{convention assume tildev is dominant and mu is flat}
		We assume $(\tilde{v},\alpha_i) \le 0$ for all $i\in Q_0 \cup \{\infty\}$, and that $\mu$ is flat. In particular, $\M_\lambda(v,w) = \M_\lambda^0(v,w)$. 
	\end{convention}
	
	\subsection{The canonical decomposition}
	Crawley-Boevey introduced a decomposition of affine quiver varieties into a product of simpler subvarieties in \cite{cb2000decomposition}. 
	More precisely, fix a quiver $Q$, a dimension vector $v$, and a parameter $\lambda \in\mathfrak{p}$. Assume there is no framing. 
	\begin{theorem}[{\cite[Theorem 1.1, Proposition 1.2]{cb2000decomposition}}]\label{CB canonical decomposition}
		There exists a decomposition $v = \sum_{i=1}^k m_i v_i$ (called the \textit{canonical decomposition} of $v$), where $m_i$ are positive integers and $v_i \in \Sigma_\lambda$, such that any  decomposition of $v$ as a sum of elements in $\Sigma_\lambda$ is a refinement of this decomposition. We have 
		\[\M_\lambda^0(v) = \prod_{i=1}^k S^{m_i} \M^0_\lambda(v_i),\]
		which we call the canonical decomposition of $\M_\lambda^0(v)$. 
		When $v_i$ is a real root, $\M^0_\lambda(v_i)$ is a point. When $v_i$ is a non-isotropic imaginary root, $m_i = 1$. 
	\end{theorem}
	We can define the canonical decomposition of $\M_\lambda^0(v,w)$ by identifying it with $\M_{\tilde{\lambda}}^0(Q^\infty,\tilde{v})$. 
	
	\begin{lemma}\label{all canonical decomposition components are framed}
		Assume \Cref{convention assume tildev is dominant and mu is flat}. Suppose $v_i$ appears in the canonical decomposition of $\M_{\tilde{\lambda}}^0(Q^\infty,\tilde{v})$. Write $v_i = \sum_{j\in Q^\infty_0} c_{ij}\alpha_j$. Then some $c_{ij} = 1$. 
	\end{lemma}
	\begin{proof}
		If $\tilde{v} \in \Sigma_{\tilde{\lambda}}$ already, then the canonical decomposition is trivial, and note the $\infty$-component of $\tilde{v}$ is 1. If not, by \cite[Section 6]{cb2000decomposition}, we are in one of the following three cases: 
		\begin{enumerate}[(I)]
			\item The quiver $Q^\infty$ is an affine type quiver and $\tilde{v} = m\delta$ for some $m\ge 2$, where $\delta$ is the minimal imaginary root. 
			\item The quiver $Q^\infty$ decomposes as in \cite[Lemma 5.3]{cb2000decomposition}. More precisely, the vertex set $Q_0^\infty$ is a disjoint union $\mathcal{J} \cup \mathcal{K}$, there is a unique arrow $a$ with one end in $\mathcal{J}$ and
			the other in $\mathcal{K}$, say connecting $j\in \mathcal{J}$ and $k\in \mathcal{K}$. Moreover, $\tilde{v}_j = \tilde{v}_k = 1$. 
			\item The quiver $Q^\infty$ decomposes as in \cite[Lemma 5.4]{cb2000decomposition}. More precisely, the vertex set $Q_0^\infty$ is a disjoint union $\mathcal{J} \cup \mathcal{K}$, there is a unique arrow $a$ with one end in $\mathcal{J}$ and
			the other in $\mathcal{K}$, say connecting $j\in \mathcal{J}$ and $k\in \mathcal{K}$. Moreover, $\tilde{v}_j  = 1$, and the restriction of $Q^\infty$ to $\mathcal{K}$ is an affine type quiver, the restriction of $\tilde{v}$ to $\mathcal{K}$ is $m\delta$ for some $m\ge 2$, where $\delta$ is the minimal imaginary root. 
		\end{enumerate} 
		Since $\tilde{v}_\infty=1$, case (I) is impossible. 
		The canonical decomposition of $\M_{\tilde{\lambda}}^0(Q^\infty,\tilde{v})$ is obtained as the product of those of $\M_{\tilde{\lambda}}^0(Q^\infty|_{\mc{J}},\tilde{v}|_{\mc{J}})$ and $\M_{\tilde{\lambda}}^0(Q^\infty|_{\mc{K}},\tilde{v}|_{\mc{K}})$. 
		In case (II), the $j$-component of $\tilde{v}|_{\mc{J}}$ and the $k$-component of $\tilde{v}|_{\mc{K}}$ are 1; 
		in case (III), the $j$-component of $\tilde{v}|_{\mc{J}}$ is 1, and $\M_{\tilde{\lambda}}^0(Q^\infty|_{\mc{K}},\tilde{v}|_{\mc{K}}) = S^m\M_{\tilde{\lambda}}^0(Q^\infty|_{\mc{K}},\delta)$, where $\delta$ also has a component 1. Inductively, we obtain \Cref{all canonical decomposition components are framed}. 
	\end{proof}
	Therefore, we may view each component $\M_{\tilde{\lambda}}^0(Q^\infty,v_i)$ as a \textit{framed} quiver variety. When $\lambda = 0$, it is clear that under \Cref{convention assume tildev is dominant and mu is flat}, the Namikawa-Weyl group of $\M_0(v,w) \cong\M_0^0(v,w)$ is the direct product of those of $S^{m_i} \M^0_0(Q^\infty,v_i)$. 
	This allows us to assume $\tilde{v} \in \Sigma_0$. 
	
	In summary, we have reduced the problem of computing the Namikawa-Weyl group of $\M_0(v,w)$ to the following situation.
	\begin{convention}\label{best convention on tilde v}
		The dimension vector $\tilde{v}\in\ZZ_{\ge 0}^{Q_0^\infty}$ is a root, lies in $\Sigma_0$, and satisfies $(\tilde{v},\alpha_i)\le 0$ for all $i\in Q_0^\infty$. The moment map $\mu:T^*R\to \g$ is flat, so that $\M_\lambda(v,w) = \M_\lambda^0(v,w)$ for $\lambda\in\mathfrak{p}$. 
	\end{convention}
	
	\begin{remarklabeled}\label{we can compute weyl group for honest affine quiver varieties} 
		Thanks to \Cref{all canonical decomposition components are framed}, we can compute the Namikawa-Weyl group of any affine quiver variety $\M_0^0(v,w)$ once we know how to compute the Namikawa-Weyl group of $\M_0(v,w)$. 
	\end{remarklabeled}
	
	\subsection{Isotropic decomposition}{\label{Deformation of conical symplectic singularities}}
	From the definition, the most straightforward way to compute the Namikawa-Weyl group of a conical symplectic singularity is to find its codimension 2 leaves, determine the corresponding slice Kleinian singularities, and examine the diagram automorphisms induced by the fundamental group action. For affine quiver varieties under \Cref{best convention on tilde v}, the first two steps are done by \cite[Definition 1.18, Theorem 1.20]{bellamy2021symplectic}. 
	\begin{theorem}\label{BS isotropic decomposition}
		Let $v\in \Sigma_\lambda$ be imaginary. Then the codimension 2 leaves of $\M_\lambda^0(v)$ corresponds to the representation types of the form
		\[\tau = (\beta^1,1;\beta^2,1;\cdots;\beta^s,1; \gamma^1,m_1;\cdots;\gamma^t,m_t)\]
		such that 
		\begin{enumerate}
			\item the $\beta^i$ are imaginary roots, and the $\gamma^i$ are pairwise distinct real roots;
			\item the slice quiver $\underline{Q} $, after removing all the edge loops, is an affine type quiver;
			\item the dimension vector $\tilde{\underline{v}} = (1,1,\cdots,1;m_1,\cdots,m_t)$ equals the minimal imaginary root of $\underline{Q} $. 
		\end{enumerate}
		The slice Kleinian singularity has the same type as the slice quiver $\underline{Q}$. 
	\end{theorem}
	The decomposition $v = \beta^1+ \cdots + \beta^s + m_1\gamma^1+ \cdots \gamma^tm_t$ is called an \textit{isotropic decomposition}.
	Note that, when $v\in\Sigma_\lambda$ is real, the dimension formula \Cref{CB strata dimension formula} implies $\M_\lambda^0(v)$ is a point. 
	
	In \cite{bellamy2021symplectic}, the action of $\pi_1(\L_i)$ on $\hat{\cartanh}_i^*$ is not computed, so we cannot directly determine $\cartanh_X$ and $W$. 
	In the remaining parts of this paper, we present two methods for computing $W$, one via tautological line bundles (\Cref{section Tautological line bundles and the Kirwan}) and the other via subgeneric deformations (\Cref{Codimension 2 leaves in deformed quiver varieties}). 

	\section{Tautological line bundles}\label{section Tautological line bundles and the Kirwan}
	In this section, we recall the notion of tautological line bundles on $\M_0^\theta(v,w)$ and describe the Namikawa-Cartan space of $\M_0^0(v,w)$ in terms of their Chern classes. We also examine what type of Weyl groups can appear as components of the Namikawa-Weyl group of $\M_0^0(v,w)$. 
	We make the assumptions of \Cref{best convention on tilde v}. 
	\subsection{Tautological line bundles}\label{subsection taut line bundles}
	Assume $\theta$ is generic so that $G$ acts freely on  $(T^*R)^{\theta-ss}$. Let $\chi\in\ZZ^{Q_0}$ be a character of $G$. 
	\begin{definition}
		The line bundle over $\M^\theta_\mathfrak{p}(v,w) $, 
		whose total space is the homogeneous bundle
		\[\mu\inverse(\mathfrak{p})^{\theta -ss} \times^G \CC_\chi\] 
		where $G$ acts on the one dimensional vector space $\CC_\chi$ by the character $\chi$, is called a \textit{tautological line bundle}. We denote it by $\mathcal{O}(\chi,\M^\theta_\mathfrak{p})$. 
		
		For any $\lambda\in \mathfrak{p}$, we denote by $\mathcal{O}(\chi,\M^\theta_\lambda)$ the restriction of $\mathcal{O}(\chi,\M^\theta_\mathfrak{p})$ to $\M^\theta_\lambda(v,w)$; its total space is \[\mu\inverse(\lambda)^{\theta -ss} \times^G \CC_\chi.\] 
		
		Similarly, we can define the line bundle $\mathcal{O}(\chi,\M^\theta_{\CC \lambda})$ on $\M^\theta_{\CC \lambda}(v,w)$. 
	\end{definition}
	
	For generic $\theta$,  $\M^\theta_\mathfrak{p}(v,w) \to\mathfrak{p}$ is a Poisson deformation of $\M^\theta_0(v,w) $. By \Cref{Namikawa diagram of universal deformation }, there is a unique linear map $\kappa: \mathfrak{p} \to H^2(\M^\theta_0(v,w),\CC ) $ such that the pullback of the universal deformation by $\kappa$ is $\M^\theta_\mathfrak{p}(v,w) \to\mathfrak{p}$. 
	The following theorem relates tautological line bundles with the map $\kappa$. Let $c_1(\L)$ denote the first Chern class of a line bundle $\L$. 
	\begin{theorem}[{\cite[Proposition 3.2.1]{losev2012isomorphisms}}]\label{kirwan map coincides with chern class}
		The map $\kappa$ coincides with the map 
		\[\ZZ^{Q_0}\to H^2(\M_0^\theta(v,w), \CC),\  \chi\mapsto c_1(\mathcal{O}(\chi,\M^\theta_0))\]
		extended by complex linearity.  
	\end{theorem}
	\begin{theorem}[{\cite[Corollary 3.8]{mcgerty2019springer}}]{\label{k is surjective}}
		The map $\kappa$ is surjective. 
	\end{theorem}
	
	Let $\chi\in\ZZ^{Q_0}$ and $\L_i$ be a codimension 2 leaf of $\M_0^0(v,w)$. Let $S_i$ denote the slice Kleinian singularity and $\tilde{S}_i$ its minimal resolution. The next proposition describes the projection of $\kappa(\chi)$ to the component $\cartanh_i$ corresponding to $\L_i$. 
	
	\begin{proposition}\label{restricting tautological line bundle to kleinian singularity resolution}
		Let 
		\[\tau = (\beta^1,1;\beta^2,1;\cdots;\beta^s,1; \gamma^1,m_1;\cdots;\gamma^t,m_t)\] 
		be the representation type of the leaf $\L_i$, 
		satisfying \Cref{BS isotropic decomposition}. Assume the $\infty$-component of $\beta^1$ is 1, and regard the vertex associated to $\beta^1$ the extended vertex of the affine type quiver. 
		The projection of $c_1(\mathcal{O}(\chi,\M^\theta_0))$ to $\cartanh_i$ is given by
		\[(\chi\cdot\beta^2,\chi\cdot \beta^3,\cdots,\chi\cdot\gamma^t) \in \cartanh_i\subset  \hat{\cartanh}^*_i \cong H^2 (\tilde{S_i},\CC). \]
	\end{proposition}
	
	\begin{proof}
		Recall \Cref{remark how to project to cartanh i}. Let $x\in\L_i$, let $U$ be a small enough analytic neighbourhood of $x$ in $\M_0^0(v,w)$ which is isomorphic to $\Delta^{\dim \M_0^0(v,w) -2}\times S_i$ and satisfies \Cref{analytic local structure theorem}. We will restrict the line bundle $\mathcal{O}(\chi,\M^\theta_0)$ to $\rho\inverse(U)\subset \M^\theta_0(v,w)$. 
		
		Let $\underline{U}$ be an analytic neighbourhood of $0\in \underline{\M}_0^0(\underline{v},\underline{w}) \times R_0$ satisfying \Cref{analytic local structure theorem}. Write
		\begin{align*}
			\pi_0^\theta : & \mu\inverse(0)^{\theta-ss} \to \M_0^\theta(v,w)\\
			\underline{\pi}_0^\theta : & G\times^H\mu_H\inverse(0)^{\theta-ss} \to \underline{\M}_0^\theta(\underline{v},\underline{w}) \times R_0
		\end{align*}
		for the projections. We see
		\[ (\pi_0^\theta)\inverse \rho\inverse (U) \cong G\times^H (\underline{\pi}_0^\theta)\inverse \underline{\rho}\inverse(\underline{U}). \]
		The total space of the restriction of $\mathcal{O}(\chi,\M^\theta_0)$ to $\rho\inverse(U)$ is 
		\begin{align}\label{local description of total space of tautological line bundle}
			\begin{split}
				(\pi_0^\theta)\inverse \rho\inverse (U) \times^G \CC_\chi &\cong (G\times^H (\underline{\pi}_0^\theta)\inverse \underline{\rho}\inverse(\underline{U}))\times^G \CC_\chi \\
				&= (\underline{\pi}_0^\theta)\inverse\underline{\rho}\inverse(\underline{U}) \times^H \CC_{\chi|_H}. 
			\end{split}
		\end{align}
		\Cref{remark how to project to cartanh i} implies that the $\cartanh_i$-component of $\kappa(\chi)$ is the Chern class of the tautological line bundle on $\underline{\M}_0^\theta(\underline{v},\underline{w})$ associated to the character $\chi|_H$ of $H$, which is exactly $(\chi\cdot\beta^2,\chi\cdot \beta^3,\cdots,\chi\cdot\gamma^t)$. 
	\end{proof}
	
	Therefore, given a codimension 2 leaf $\L_i$ and the corresponding representation type $\tau$, we can determine $\cartanh_i$ thanks to \Cref{k is surjective} and \Cref{kirwan map coincides with chern class}. 
	\subsection{Examples}
	We give examples of $\M_0^0(v,w)$ whose Namikawa-Weyl groups have type $B_n$ and $G_2$, and show that type $C_n ,n\ge 3$ and $F_4$ Weyl groups cannot appear as a component of the Namikawa-Weyl group. 
	
	\begin{example}
		By \Cref{MN weyl group dynkin quiver}, any $A,D$ or $E$ type Weyl group can be the Namikawa-Weyl group of some $\M_0^0(v,w)$. 
	\end{example}
	The following example is suggested to the author by Hiraku Nakajima, and it appears in \cite[5(v)]{braverman2017ring}. 
	\begin{example}\label{Nakajima G2 quiver via isotropic decomposition}
		Consider the following quiver:
		\begin{center}
			\begin{tikzpicture}[scale=0.9]
				\filldraw[black] (0,0) circle (1.5pt);
				\filldraw[black] (1,0) circle (1.5pt);
				\filldraw[black] (2,0) circle (1.5pt);
				\node at (2,0) (alpha3) {$\ $};
				\draw[-,thick] (0.1,0)--(0.9,0);
				\draw[-,thick] (1.1,0)--(1.9,0);
				\draw [-,thick] (2.1,-0.15)arc(-160:160:0.5 and 0.3);
				\node at (0,-0.4) {$\alpha_\infty$};	
				\node at (1,-0.4) {$\alpha_1$};
				\node at (2,-0.4) {$\beta$};
				\node at (0,0.4) {$1$};	
				\node at (1,0.4) {$2$};
				\node at (2,0.4) {$3$};
			\end{tikzpicture}
		\end{center}
		where we view $\alpha_\infty$ as a framing, so that $\tilde{v} = \alpha_\infty+2\alpha_1+3\beta$. It is not hard to see $\Tilde{v}\in\Sigma_0$, and the only codimension 2 leaf corresponds to the representation type 
		\[\tau = (\alpha_\infty, 1; \alpha_1,2; \beta,1;\beta,1;\beta,1).\]
		For any $x\in\L_\tau$, the (extended) slice quiver $\underline{Q}^\infty$ is the type $\hat{D}_4$ quiver, and the dimension vector $\tilde{\underline{v}}$ is precisely the minimal imaginary root type $\hat{D}_4$. The image of $\kappa: \mf{p} \to \hat{\cartanh}^*$ is exactly the subspace of fixed points of the degree $3$ diagram automorphism, by \Cref{restricting tautological line bundle to kleinian singularity resolution}. Therefore, the Namikawa-Weyl group of $\M_0^0(\Tilde{v})$ is the Weyl group of type $G_2$. 
	\end{example}
	\begin{example}\label{Bk example}
		For any $k\ge 0$, consider the following quiver:
		\begin{center}
			\begin{tikzpicture}[scale=0.9]
				\filldraw[black] (0,-0.6) circle (1.5pt);
				\filldraw[black] (0,0.6) circle (1.5pt);
				\filldraw[black] (1,0) circle (1.5pt);
				\filldraw[black] (2,0) circle (1.5pt);
				\filldraw[black] (3,0) circle (1.5pt);
				\filldraw[black] (4,0) circle (1.5pt);
				\draw[-,thick] (0.1,0.6)--(0.9,0.1);
				\draw[-,thick] (0.1,-0.6)--(0.9,-0.1);
				\draw[-,thick] (1.1,0)--(1.9,0);
				\draw[-,thick] (3.1,0)--(3.9,0);
				\draw [-,thick] (4.1,-0.15)arc(-160:160:0.5 and 0.3);
				\node at (2.5,0) {$\cdots$};	
				\node at (-0.5,0.6) {$\alpha_\infty$};	
				\node at (-0.5,-0.6) {$\alpha_0$};
				\node at (1,-0.4) {$\alpha_1$};
				\node at (2,-0.4) {$\alpha_2$};
				\node at (3,-0.4) {$\alpha_{k}$};
				\node at (4,-0.4) {$\beta$};
				\node at (0,1) {$1$};	
				\node at (0,-1) {$1$};	
				\node at (1,0.4) {$2$};
				\node at (2,0.4) {$2$};
				\node at (3,0.4) {$2$};
				\node at (4,0.4) {$2$};
			\end{tikzpicture}
		\end{center}
		We view $\alpha_\infty$ as a framing so ${v}=  \alpha_0 + 2(\alpha_1+\cdots + \alpha_k) + 2 \beta$. 
		A direct computation shows that 
		$\Sigma_0 = \{\alpha_i, \beta, \tilde{v}\}$, $i\in\{0,1,\cdots,k,\infty\}$, and there is only one codimension 2 leaf of $\M_0^0(v)$, corresponding to the representation type 
		\[\tau = (\alpha_\infty,1;\alpha_0,1;\alpha_1,2;\cdots, \alpha_k,2;\beta,1;\beta,1).\]
		The slice quiver $\underline{Q}^\infty$ has type $\hat{D}_{k+3}$, and it is easy to see the image of $\kappa$ is the subspace stable under the order 2 diagram folding. Therefore, the Namikawa-Weyl group of $\M_0^0(\tilde{v})$ is the Weyl group of type $B_{k+2}$. 
	\end{example}
	\begin{proposition}\label{F4 cannot appear}
		$F_4$ cannot appear as a component of the Namikawa-Weyl group. 
	\end{proposition}
	\begin{proof}
		Assume $F_4$ Weyl group appears. 
		Recall that $F_4$ is obtained by folding the type $E_6$ diagram. Therefore, by \Cref{BS isotropic decomposition}, the slice quiver $\underline{Q}^\infty$ is $\hat{E}_6$, and the dimension vector $\tilde{\underline{v}}$ is the minimal imaginary root $\delta$. 
		Under a suitable labeling, the root $\delta$ takes the form $(1,2,3,2,1,2,1)$, and we require the image of $\kappa$ to be pointwise fixed by the order 2 folding. In particular, the first two ``2" in the expression of $\delta$ must both correspond to the same real root in  $\Sigma_0(Q^\infty,\tilde{v})$, contradicting \Cref{BS isotropic decomposition}.  
	\end{proof}
	\begin{proposition}\label{Cn cannot appear}
		For $n\ge 3$, $C_n$ cannot appear as a component of the Namikawa-Weyl group. 
	\end{proposition}
	\begin{proof}
		Suppose $\L_i$ is a codimension 2 leaf such that $\cartanh_i$ has type $C_n$. Then $\hat{\cartanh}_i^*$ has type $A_{2n-1}$ or $A_{2n}$.
		Suppose first $\hat{\cartanh}_i^*$ has type $A_{2n-1}$. Then the slice quiver $\underline{Q}^\infty$ is $\hat{A}_{2n-1}$, and dimension vector $\tilde{\underline{v}}$ is the minimal imaginary root $\delta$. Let the vertices of $\hat{A}_{2n-1}$ be denoted by $\alpha_0, \alpha_1,\alpha_2,\cdots, \alpha_{n-1}, \beta, \gamma_{n-1},\cdots,\gamma_1$, where $\alpha_0$ corresponds to the extended vertex, $\alpha_j,\gamma_j$ and $\beta$ are elements of $\Sigma_0$, and the representation type of $\L_i$ is 
		\[\tau = (\alpha_0,1;\alpha_1,1;\cdots;\gamma_1,1).\]
		The diagram automorphism identifies $\alpha_i$ and $\gamma_i$. Since ${\cartanh}_i $ has type $C_n$, $\alpha_j=\gamma_j$  and they are all imaginary roots. In particular, $2\alpha_j$ are roots. 
		
		Claim: $\L_i$ cannot have codimension 2. 
		Let us compute $\dim \L_i$ first:
		\begin{align}
			\begin{split}
				\dim \L_i =d(\tau) &= 2-(\alpha_0,\alpha_0) + 2-(\beta,\beta) + 2\sum_{j\ge 1} (2-(\alpha_j,\alpha_j))\\
				&= 4n -( (\alpha_0,\alpha_0) + (\beta,\beta) + 2\sum_{j=1}^{n-1} (\alpha_i,\alpha_i))
			\end{split}
		\end{align}
		On the other hand, by assumption $\Tilde{v} = \alpha_0+\beta  + 2\sum \alpha_j  \in \Sigma_0$. Therefore,
		\begin{align}
			\begin{split}
				\dim \M_0^0(Q,v,w) & =2p(\Tilde{v})   \\ & = 2- ( (\alpha_0,\alpha_0) + (\beta,\beta) + 4\sum_{i=1}^{n-1} (\alpha_i,\alpha_i)) \\
				& +(\alpha_0,2\alpha_1) + (2\alpha_1,\alpha_0+2\alpha_2)   \\
				& + (2\alpha_2,2\alpha_1+2\alpha_3)\\
				& \vdots \\
				& + (2\alpha_{n-1},2\alpha_{n-2} + \beta) + (\beta,2\alpha_{n-1})\\
				& =2 - ( (\alpha_0,\alpha_0) + (\beta,\beta) + 4\sum_{i=1}^{n-1} (\alpha_i,\alpha_i)) + 8(n-1) 
			\end{split}
		\end{align}
		Therefore, $\L_i$ has codimension 2 if and only if $n=2$ and $(\alpha_1,\alpha_1) = 0$ (this is exactly the situation of \Cref{Bk example} when $k=0$). 
		
		The proof for the case where $\underline{Q}^\infty$ has type $\hat{A}_{2n}$ is completely analogous. 
		We conclude that $C_n$ for $n\ge 3$ cannot appear as a Namikawa-Weyl group component of $\M_0^0(v,w)$.
	\end{proof}
	\section{Subgeneric deformations}\label{Codimension 2 leaves in deformed quiver varieties}
	It is not easy to classify codimension 2 leaves of $\M_0^0(v,w)$ (equivalently, find all the isotropic decompositions \Cref{BS isotropic decomposition}), although once this is done, it is straightforward to calculate the Namikawa-Weyl group by methods in 
	\Cref{section Tautological line bundles and the Kirwan}. 
	We will take a different approach to determine the Namikawa-Weyl groups, one via subgeneric deformation. 
	We will see the relation between subgeneric deformations and isotropic decompositions in \Cref{indecomposable cod 2 roots give iso decomposition}. 
	
	In this subsection, we examine the deformations $\M_\lambda^0(v,w)$ of $\M_0^0(v,w)$ where $\lambda$ is \textit{subgeneric}, i.e. very close to being generic. In particular, we examine when such a deformation have a codimension 2 symplectic leaf. We continue to assume \Cref{best convention on tilde v}. 
	
	\subsection{The set $\Sigma_{\Tilde{\lambda}}$ for subgeneric $\lambda$}
	\begin{definition}
		Let $\lambda\in \mathfrak{p}$ and $v^1 \le v$ be a dimension vector in $\ZZ_{\ge0}^{Q_0}$. 
		\begin{enumerate}
			\item We say $v^1$ is \textit{indecomposable with respect to $\lambda$} if there is no proper decomposition $v^1 = \sum \beta_i$ where all $\beta_i$ are positive roots such that $\beta_i\cdot \lambda = 0$. 
			\item We say $\lambda$ is \textit{subgeneric for }$v$ if there is a unique (up to scalar) positive root $v^1$ of $Q$ such that $v^1<v$ and $\lambda\cdot v^1 = 0$. We just say $\lambda$ is \textit{subgeneric} when $v$ is clear from context.
		\end{enumerate}  
	\end{definition}
	Recall the definitions of $\tilde{v},\tilde{\lambda}$ in \Cref{equivalent description}. 
	If $\lambda$ is subgeneric and $v^1$ is indecomposable with respect to $\lambda$, then  $v^1\in\Sigma_{\Tilde{\lambda}}$ if we view it as an element of $\CC^{Q^\infty_0}$ whose $\infty$-component is 0.
	The following lemmas are used in the proof of \Cref{theorem cod 2 leaves in deformed quiver varietyes}.
	\begin{lemma}{\label{Lemma roots nv1 in Sigma for subgeneric lambda}}
		Let $\lambda$ be subgeneric and $v^1$ be the indecomposable root with respect to $\lambda$. 
		Let $n\in \ZZ_{>0}$. 
		\begin{enumerate}
			\item If $v^1$ is a real root, then $nv^1\in\Sigma_{\Tilde{\lambda}}$ if and only if $n=1$. 
			\item If $v^1$ is an isotropic imaginary root, i.e. $(v^1,v^1) = 0$, then $nv^1\in\Sigma_{\Tilde{\lambda}}$ if and only if $n=1$. 
			\item  If $v^1$ is a non-isotropic imaginary root, i.e. $(v^1,v^1) \le -2$, then $nv^1\in\Sigma_{\Tilde{\lambda}}$ for all $n$.
			
		\end{enumerate}
		
	\end{lemma}
	\begin{proof}
		The assumption on $v^1$ guarantees that for $n\ge 2$, any decomposition of $nv^1$ in (2) of \Cref{CB simple dimension criteria} has the form $nv^1 = \sum_{i=1}^k n_iv^1$ where $\sum_{i=1}^k n_i = n$.

		The first two parts follow easily from \Cref{CB simple dimension criteria}. 
		Let $v^1$ be a non-isotropic imaginary root.
		We have \[p(nv^1) = 1-\frac{n^2}{2}(v^1,v^1),\ \  \sum_{i=1}^k p(n_iv^1) = k-\sum_{i=1}^k \frac{n_i^2}{2}(v^1,v^1).\] 
		Since $(v^1,v^1) \le -2$, it is clear that, for $k\ge 2$ and $\sum_{i=1}^k n_i=n$,
		\begin{equation}\label[ineq]{inequality useful for non isotropic p}
			1-\frac{n^2}{2}(v^1,v^1) > k-\sum_{i=1}^k \frac{n_i^2}{2}(v^1,v^1),
		\end{equation}
		i.e. $p(nv^1) > \sum p(n_iv^1)$, 
		which proves part (3) of the lemma. 
	\end{proof}
	If $\lambda$ is subgeneric for $v$ and $v^1$ is the indecomposable root with respect to $\lambda$, then any $u \in \Sigma_{\tilde{\lambda}}$ satisfying $u\le \tilde{v}$ have the form $ u= nv^1$ or $ u= \tilde{v}-nv^1$ for some $n\ge 0$. The following lemma gives some conditions for  $\tilde{v} - nv^1 \in \Sigma_{\Tilde{\lambda}}$. 
	
	\begin{lemma}{\label{Lemma roots tildev-nv1 for subgeneric lambda}}
		Assume $\M_0^\theta(v,w)$ is nonempty for generic $\theta$. Let $\lambda\in\mathfrak{p}$ be subgeneric for ${v}$, and $v^1$ be the unique indecomposable positive root such that $v^1 \le v$ and $v^1\cdot \lambda = 0$ (so that $v^1\in\Sigma_{\Tilde{\lambda}}$). 
		Let $n\in \ZZ_{\ge 0}$. 
		\begin{enumerate}[(1)]
			\item Suppose $v^1$ is a real root. Then $\tilde{v}- nv^1 \in\Sigma_{\Tilde{\lambda}}$ if and only if $\tilde{v}- nv^1$ is a root and $(\tilde{v}- nv^1,v^1)\le 0$. 
			\item Suppose $v^1$ is an isotropic imaginary root. If $(\tilde{v},v^1) = 0$ or $-1$, then $\tilde{v}-nv^1 \in \Sigma_{\Tilde{\lambda}}$ if and only if $n$ is the largest integer such that $\tilde{v} - nv^1$ is a root. If $(\tilde{v},v^1) \le -2$, then $\tilde{v} \in \Sigma_{\Tilde{\lambda}}$.
			\item Suppose $v^1$ is a non-isotropic imaginary root, and let $N$ be the largest integer such that $\tilde{v}-Nv^1$ is a root. Then $\tilde{v}-nv^1 \in \Sigma_{\Tilde{\lambda}}$ if and only if $n=N$ or $p(\tilde{v}-nv^1)>p(\tilde{v}-Nv^1)+p((N-n)v^1)$. 
		\end{enumerate}
	\end{lemma}
	\begin{proof}
		\begin{enumerate}[(1)]
			\item Let $v^1$ be a real root. Assume $\tilde{v}-nv^1$ is a root, and $(\Tilde{v} - nv^1,v^1) = c > 0$. Let $s$ be the reflection along the real root $v^1$. Then $s(\Tilde{v} - nv^1) = \Tilde{v} - (n+c)v^1$ is a root. 
			But since $(\tilde{v} - nv^1,\tilde{v} - nv^1 ) = (s(\tilde{v} - nv^1),s(\tilde{v} - nv^1))$, and $p(v^1) = 0$,  we have 
			\[p(\tilde{v}-(n+c)v^1) + c p(v^1) = p(\Tilde{v} - nv^1).\]
			Therefore $\tilde{v}-nv^1 \not\in \Sigma_{\tilde{\lambda}}$. 
			
			Conversely, suppose $(\tilde{v}-nv^1,v^1)\le 0$ and $\tilde{v}-nv^1$ is a root. Let $(\tilde{v} - nv^1 - lv^1) + v^1 + ... + v^1$ be a decomposition of $\tilde{v}-nv^1$, $l>0$, and note that any decomposition in (2) of \Cref{CB simple dimension criteria} has this form. We have
			\begin{align*}
				p(\tilde{v} - nv^1 - lv^1) + lp(v^1) &= p(\tilde{v}-nv^1) + (\tilde{v}-nv^1,lv^1) - \frac{1}{2}l^2(v^1,v^1) + 0\\
				&=p(\tilde{v}-nv^1) + l(\tilde{v}-nv^1,v^1) - l^2\\
				&< p(\tilde{v} - nv^1).
			\end{align*}
			Therefore $\tilde{v} - nv^1 \in \Sigma_{\tilde{\lambda}}$.
			
			\item Let $v^1$ be an isotropic imaginary root. Suppose $(\tilde{v},v^1) = 0$ or $-1$. If $n$ is the largest integer such that $\tilde{v} - nv^1$ is a root then clearly $\tilde{v} - nv^1\in\Sigma_{\Tilde{\lambda}}$, since condition (2) of \Cref{CB simple dimension criteria} is automatic. 
			Conversely, suppose there is $m>n$ such that $\tilde{v} - mv^1$ is a root. 
			Then 
			\begin{align*}
				p(\tilde{v} - nv^1) &= p(\tilde{v} - mv^1 + (m-n) v^1) \\
				&= p(\tilde{v} -mv^1) - (m-n)(\tilde{v} - mv^1,v^1) -0\\
				& \le p(\tilde{v} - mv^1) + (m-n) \\
				&= p(\tilde{v} - mv^1) +(m-n)p(v^1).
			\end{align*} Therefore, $\tilde{v} - nv^1\not\in\Sigma_{\Tilde{\lambda}}$.
			This proves the first half of (2). 
			
			Now assume $(\tilde{v},v^1) \le -2$. For any $n\ge 1$ and any decomposition $\tilde{v} = (\tilde{v}-nv^1)+n_1v^1+\ldots+n_kv^1$ where $\sum_{i=1}^k n_i= n$, 
			\[p(\tilde{v} - nv^1) + \sum_i p(n_i v^1) = p(\tilde{v}) + n (\tilde{v},v^1) + k < p(\tilde{v}).\] 
			Since $\tilde{v}$ is a root by \Cref{Nonempty smooth quiver variety implies root}, $\tilde{v}\in \Sigma_{\Tilde{\lambda}}$. 
			
			\item It is clear from definition that $\tilde{v}-Nv^1\in\Sigma_{\Tilde{\lambda}}$, and if $\tilde{v}-nv^1\in \Sigma_{\Tilde{\lambda}}$ then $p(\tilde{v}-nv^1)>p(\tilde{v}-Nv^1)+p((N-n)v^1)$. 
			
			Conversely, suppose $n<N$ and $p(\tilde{v}-nv^1)>p(\tilde{v}-Nv^1)+p((N-n)v^1)$. By \Cref{When is p(a+b)>p(a)+p(b)}, this inequality holds if and only if $- (N-n)(\tilde{v}-Nv^1,v^1)>1$, which implies $(\tilde{v}-Nv^1,v^1) \le -1$. 
			
			Let $N>m>n$ be such that $\tilde{v}-mv^1$ is a root. 
			Note that 
			\[(\tilde{v}-mv^1,v^1) = (\tilde{v}-Nv^1,v^1)+(N-m)(v^1,v^1) < (\tilde{v}-Nv^1,v^1) \le -1\]
			and therefore $- (m-n)(\tilde{v}-mv^1,v^1)>1$. So $p(\tilde{v}-nv^1)>p(\tilde{v}-mv^1)+p((m-n)v^1)$ by \Cref{When is p(a+b)>p(a)+p(b)}. Thanks to \Cref{inequality useful for non isotropic p}, we conclude $\tilde{v}-nv^1\in \Sigma_{\Tilde{\lambda}}$. 
		\end{enumerate}
	\end{proof}
	\subsection{Codimension 2 roots}
	In this subsection we investigate for which subgeneric $\lambda \in \mathfrak{p}$ the variety $\M_\lambda^0(v,w)$ have a codimension 2 symplectic leaf. Let $v^1\in \Sigma_{{\lambda}}$ be the minimal positive root such that $\lambda \cdot v^1 = 0$. 
	
	\begin{theorem}{\label{theorem cod 2 leaves in deformed quiver varietyes}}
		The affine variety $\M_\lambda^0(v,w) = \M_{\tilde{\lambda}}^0(\tilde{v})$ has a codimension 2 symplectic leaf if and only if one of the following conditions hold: 
		
		\begin{enumerate}
			\item $v^1$ is a real root,  $\<\nu, v^1\> = 0$, and $\tilde{v} - v^1$ is a root of ${Q^\infty}$.
			\item $v^1$ is an isotropic imaginary root, $\<\nu ,v^1\>=2$, and $ \tilde{v} - v^1$ is a root of ${Q^\infty}$.  
			\item $v^1$ is a non-isotropic imaginary root, $\tilde{v}-n v^1$ is a root for $n=1$ or $2$,  $(\tilde{v} - nv^1, nv^1) = -2$, and $\tilde{v}-mv^1$ is not a root for any $m> n$.
		\end{enumerate}
	\end{theorem}
	
	\begin{definition}\label{definition codimension 2 root}
		We say an indecomposable root $v^1$ of $Q$ is a \textit{codimension 2 root} if it satisfies one of the conditions in \Cref{theorem cod 2 leaves in deformed quiver varietyes}. We say it is of type $(i)$ if it falls in Case $(i)$ in \Cref{theorem cod 2 leaves in deformed quiver varietyes}, $i=1,2,3$. 
	\end{definition}
	
	The following lemma and its corollary on imaginary roots are useful. It was proved for quivers without loops, but the proof works for quivers with loops as well, thanks to \cite[Lemma 5,Lemma 6]{kavc1980some}.

	\begin{lemma}[{\cite[Theorem 2]{morita1980roots}}]{\label{lemma imaginary root}}
		Let $Q$ be a quiver and $\alpha = \sum n_i\alpha_i\neq 0, n_i\in\ZZ_{\ge 0}$. Then $\alpha$ is an imaginary root if and only if for any $\sigma\in W_Q$, we have $\sigma(\alpha) > 0$ and $\Supp(\alpha)$ is connected. \qed
	\end{lemma}
	
	\begin{corollary}{\label{sum of imaginary root is root}}
		Suppose $v$ is a positive root, $\alpha$ is a positive imaginary root, and $(v,\alpha)\le -1$, then $v+\alpha$ is a positive imaginary root. 
	\end{corollary}
	\begin{proof}
		If $v$ is real, then the $v$-string through $\alpha$ contains $\alpha+v$ since $(v,\alpha)\le -1$. Since $(v+\alpha,v+\alpha) = (v,v) + 2(v,\alpha) + (\alpha,\alpha) \le 0$,  $v+\alpha$ must be an imaginary root. 
		
		If $v$ is imaginary, then for any $w\in W_Q$, $w.\alpha$ and $w.v$ are positive with connected support. Since $(w.\alpha, w.v) = (\alpha,v) \neq 0$, we conclude $w.(v+\alpha)$ has connected support. Therefore, $v+\alpha$ is an imaginary root by \Cref{sum of imaginary root is root}. 
	\end{proof}
	
	\begin{proof}[Proof of \Cref{theorem cod 2 leaves in deformed quiver varietyes}]
		We use \Cref{CB strata dimension formula} to search for codimension 2 strata. Recall, \Cref{CB strata dimension formula}, that $d(\tau)$ denotes the dimension of the stratum associated to a representation type $\tau$. Moreover, by the assumption $\tilde{v} \in \Sigma_0$, we see $\tilde{v}\in\Sigma_{\tilde{\lambda}}$ for all $\lambda\in\mf{p}$. Therefore, $\dim \M_\lambda^0(v,w) = 2p(\tilde{v})$. 
		\begin{enumerate}
			\item If $v^1$ is real, then by \Cref{Lemma roots nv1 in Sigma for subgeneric lambda}, any possible representation type has the form $\tau_n = (\Tilde{v}- nv^1,1;v^1,n)$. We have
			\[d(\tau_n) =
			2(p(\Tilde{v} - nv^1) + p(v^1)) 
			= 2( p(\Tilde{v}) + n(\Tilde{v}, v^1) - n^2).\] 
			Note this formula holds for $n=0$ as well. 
			Then $\dim \M_\lambda^0(v,w)-d(\tau_n) = 2n^2-2n(\tilde{v},v^1)$. If the leaf associated to $\tau_n$ has codimension 2 if and only if $\Tilde{v}- nv^1 \in\Sigma_{\tilde{{\lambda}}}$ and $n(n-(\tilde{v} , v^1)) = 1$, i.e. 
			\[n=1,(\tilde{v} , v^1)=0. \]
			Conversely, if $(\tilde{v},v^1) = 0$ and $\tilde{v}-v^1$ is a root, then by \Cref{Nonempty smooth quiver variety implies root} and \Cref{Lemma roots tildev-nv1 for subgeneric lambda}, $\tilde{v}-v^1\in \Sigma_{\Tilde{\lambda}}$. This proves case (1). 
			
			\item Suppose $v^1$ is an isotropic imaginary root, i.e. $(v^1,v^1) = 0$. By \Cref{Lemma roots nv1 in Sigma for subgeneric lambda}, $nv^1 \in \Sigma_{\Tilde{\lambda}}$ if and only if $n=1$. Therefore, any valid representation type has the form $\tau = (\tilde{v}-nv^1,1;v^1,n_1;...;v^1,n_k)$ where $n = n_1+ ... + n_k$. By \Cref{CB strata dimension formula}, the dimension of the stratum corresponding to $\tau$ is
			\begin{equation}{\label{dim formula for isotropic strata}}
				d(\tau) =2(p(\tilde{v} - nv^1) + k) = 2(p(\tilde{v}) +n(\tilde{v}, v^1)+k).
			\end{equation}
			Therefore, for fixed $n$, the dimension of stratum is maximal when $k=n$ and all $n_k=1$; in this case, $d(\tau) = 2p(\tilde{v}) + 2n(1+(\tilde{v},v^1))$.
			The dimension difference is 
			\[\dim \M_\lambda^0(v,w) - d(\tau) = -2n(1+(\tilde{v},v^1)).\]
			The difference is 2 if and only if 
			\[n=1, (\tilde{v},v^1) = -2. \]
			This proves part (2) of the theorem. 

			\item Suppose $v^1 \in \Sigma_{\tilde{\lambda}}$ is a non-isotropic imaginary root, i.e. $(v^1,v^1) \le -2$. 
			By \Cref{Lemma roots nv1 in Sigma for subgeneric lambda}, $nv^1 \in \Sigma_{\tilde{\lambda}}$ for all integers $n\ge 1$. 	
			Any representation type has the form $\tau = (\tilde{v}-nv^1,1; n_1v^1,m_1;...;n_kv^1,m_k)$ where $\tilde{v} - nv^1\in\Sigma_{\Tilde{\lambda}}$ and $\sum_i m_in_i = n$. The dimension of the corresponding stratum is \[d(\tau) = 2(p(\tilde{v}-nv^1)+\sum_ip(n_iv^1)) = 2p(\tilde{v}-nv^1)+2k-\sum_i n_i^2 (v^1,v^1). \]For fixed $n$, the dimension is maximal when $k=1$. 
			For $n\ge1$, we denote 
			\[\tau_n = (\tilde{v} - nv^1,1;nv^1,1),\] 
			so that 
			\begin{equation}{\label{dtaun for non isotropic v1}}
				d(\tau_n) = 2(p(\tilde{v}) + n(\tilde{v},v^1) + 1 - n^2 (v^1,v^1) ).
			\end{equation}
			The dimension difference is 
			\[\dim \M_\lambda^0(v,w) - d(\tau_n) = -2(n(\tilde{v},v^1) + 1 - n^2 (v^1,v^1)).\]
			The difference is 2 if and only if $n(\tilde{v}-nv^1,v^1) = -2$.
			Therefore, either $n=1$ and $(\tilde{v}-v^1,v^1)= -2$; or $n=2$ and $(\tilde{v}-2v^1,2v^1)=-2$. 
			
			In either case, if $\tilde{v} - m v^1$ is a root for some $m>n$, then it is easy to see $d(\tau_m)>d(\tau_n)$, contradicting that $\tilde{v}\in\Sigma_{\tilde{{\lambda}}}$. 
			This proves part (3) of the theorem. 
		\end{enumerate}
	\end{proof}
	
	\begin{example}{\label{examples of theorem 3.1}}
		We give an example of each of the above cases. 
		\begin{enumerate}
			\item Let $Q$ be the type $A_1$ quiver, with one vertex and no loops. Let $v = 1,w=2$. Let $\alpha$ be the simple root. The variety $\M_0(v,w) =\Spec(\CC[T^*\PP^1]) \cong \mc{N}$, the nilpotent cone of $\sl_2$. We can take $v^1 = \alpha$. This is an example of case (1). 
			\item Let $Q$ be the type $\hat{A}_2$ quiver, $v = \delta$ and $w$ be $2$ at the extending vertex, $0$ at the other two vertices. Then $v^1 = \delta$ is an example of case (2). 
			\item Let $Q$ be the following quiver, $v = \alpha + 2\beta$ and the framing be 1 over $\beta$. 
			\begin{center}
				\begin{tikzpicture}[scale=0.9]
					\filldraw[black] (1,0) circle (1.5pt);
					\filldraw[black] (2,0) circle (1.5pt);
					\node at (2,0) (alpha3) {$\ $};
					\draw[-,thick] (1.1,0)--(1.9,0);
					\draw[-,thick] (2,0.1)--(2,0.9);
					\draw [-,thick] (2.1,-0.15)arc(-160:160:0.5 and 0.3);
					\node at (1,-0.4) {$\alpha$};
					\node at (2,-0.4) {$\beta$};
					\node at (2,1.3) [rectangle,draw]  {1};
				\end{tikzpicture}
			\end{center}
			Then, $v^1 = \alpha + 2\beta$ is an example of case (3) with $n=1$. Note that this is a special case of \Cref{Bk example} for $k=0$. 
			\item Let $Q$ be the quiver with 1 vertex and 2 edge loops; let $v = 2, w = 1$. Then $v$ is an example of case (3) with $n=2$. 

		\end{enumerate}
	\end{example}
	
	\begin{remarklabeled}{\label{remark that theorem cod 2 leaves in deformed quiver varietyes is bs cor}} 
		\Cref{theorem cod 2 leaves in deformed quiver varietyes} is an example of \Cref{BS isotropic decomposition}. In each case, the slice quiver $\underline{Q}^\infty$ (without edge loops) for a codimension 2 leaf is type $\hat{A}_1$, and the dimension vector $\tilde{\underline{v}}$ is $(1,1)$, the minimal imaginary root. Moreover, it is clear from the proof that in each case $\M_\lambda^0(v,w)$ has exactly one codimension 2 leaf. 
	\end{remarklabeled}
	
	\subsection{Restriction of tautological line bundles}
	Let $v^1$ be a codimension 2 root of $Q$, $\lambda$ a subgeneric parameter such that $\lambda\cdot v^1 = 0$. The variety $\M_\lambda^0(v,w)\cong \M_\lambda(v,w)$ has a unique codimension 2 symplectic leaf $\mc{L}$. Let $x\in\mc{L}$ and consider the natural map $\rho: \M^\theta_\lambda(v,w) \to \M_\lambda^0(v,w) $. 
	\begin{lemma}\label{lemma degree of line bundle on slice P1}
		Assume $\theta\cdot v^1 > 0$. Then the fiber $\rho\inverse(x) $ is identified with $\PP^1$, and the restriction of $\mathcal{O}(\chi,\M^\theta_\lambda)$ to $\rho\inverse(x) $ has degree $\chi\cdot v^1$. 
	\end{lemma}
	\begin{proof}
		The quiver $\underline{Q}$ defined with respect to $x\in \mathcal{L}$ in \Cref{Structure of neighbourhood} can be viewed as a type $A_1$ quiver, with dimension vector 1 and framing 2. Therefore, the fiber is $\rho\inverse(x)\cong \underline{\rho}\inverse(0) \cong \PP^1$ is the zero section of $T^*\PP^1$. 
		Thanks to \Cref{analytic local structure theorem}, we can pick an analytic neighbourhood $U,\underline{U}$ as in the proof of \Cref{restricting tautological line bundle to kleinian singularity resolution},
		so that the restriction of  $\mathcal{O}(\chi,\M^\theta_\lambda)$ to $U \cong \underline{U}$ is \[(R_0 \times \mu_H\inverse(0)^{\theta-ss}\times^H \CC_{\chi|_H})|_{\underline{U}}.\] 
		We have $H \cong \CC^*$ acting on $\CC$ by the character $\chi\cdot v^1$. Therefore, the restriction of  $\mu_H\inverse(0)^{\theta-ss}\times^H \CC_{\chi|_H}$ to $\underline{\rho}\inverse(0) \cong\PP^1$ is the line bundle $\sheafO(\chi\cdot v^1)$. 
	\end{proof}
	We remark that if we assumed $\theta\cdot v^1<0$ instead, then the restriction of $\mathcal{O}(\chi,\M^\theta_\lambda)$ to $\rho\inverse(x) $ has degree $-\chi\cdot v^1$. 
	
	\section{Codimension 2 leaves in deformations of conical symplectic singularities}{\label{Codimension 2 leaves in deformed conical symplectic singularity}}
	Let $X$ be a conical symplectic singularity and  $\pi: \mc{X} \to \cartanh_X/W$ be the universal graded Poisson deformation as defined in \Cref{Deformation of conical symplectic singularities}. 
	Let $\lambda\in \cartanh_X$ and $\bar{\lambda}$ be its image in $\cartanh_X/W$. We investigate when does $\mc{X}_\lambda = \pi\inverse(\bar{\lambda})$ have a codimension 2 symplectic leaf. 
	
	\subsection{Deformation of Kleinian singularities}{\label{subsection Deformation of Kleinian singularities}}
	In this subsection we classify the deformation parameters of a Kleinian singularity that gives codimension 2 symplectic leaves in the deformed variety. We follow \cite{crawley1998noncommutative}. 
	
	Let $\Gamma \subset \SL_2(\CC)$ be a finite subgroup. Let $Q$ be the corresponding McKay quiver, which is an affine type quiver. Its vertex set $Q_0$ is in bijection with $\{N_i\}$, the set of simple representations of $\Gamma$.
	Let $\delta$ be the minimal positive imaginary root of this quiver. Then the quiver variety $\M_0^0(Q,\delta)$ is isomorphic to $X = \CC^2/\Gamma$. 
	
	Let $\lambda \in (\CC\Gamma)^\Gamma$ (where $\Gamma$ acts on $\CC\Gamma$ by conjugation). We can identify $\lambda $ with $(\tr \lambda|_{N_i}) $ and view it as an element of $\mathfrak{p}$. Define the algebra $\mathscr{I}^\lambda = (\CC\<x,y\>\# \Gamma )/(xy-yx-\lambda)$ where $\#$ denotes the skew group algebra. Then define  $\mathscr{O}^\lambda = e\mathscr{I}^\lambda e$ where $e = \frac{1}{|\Gamma|} \sum_{\gamma \in \Gamma} \gamma$. 
	
	\begin{theorem}[{\cite[Theorem 0.2]{crawley1998noncommutative}}]{\label{CBH0.2}}
		Suppose $\lambda\cdot \delta = 0$, then the algebra $\mathscr{O}^\lambda$ is commutative, and $\Spec (\mathscr{O}^\lambda) = \M_\lambda^0(Q,\delta)$. 
	\end{theorem}
	We say a root $\alpha$ of $Q$ is a \textit{Dynkin root} if $N_0\not\in \Supp(\alpha)$. 
	\begin{theorem}[{\cite[Theorem 0.4]{crawley1998noncommutative}}]{\label{CBH0.4}}
		If $\lambda\cdot \delta = 0$, then the global dimension of $\mathscr{O}^\lambda$ is finite if and only if $\lambda\cdot \alpha \neq 0$ for any Dynkin root $\alpha$ of $Q$.
	\end{theorem}
	The following corollary now follows from the fact that a Noetherian local ring is regular if and only if its global dimension is finite. 
	\begin{corollary}\label{Corollary of CBH0.4 on smoothness}
		Suppose $\lambda\cdot \delta = 0$, then $\M_\lambda^0(Q,\delta)$ is smooth if and only if $\lambda \cdot \alpha \neq 0$ for any Dynkin root $\alpha$ of $Q$. \qed 
	\end{corollary}
	
	Let $\g$ be the simple Lie algebra whose Dynkin diagram is $Q_{\fin}$ (that is $Q$ without the extended vertex), $\cartanh \subset \g$ be a Cartan subalgebra, and $W$ be the Weyl group of $\g$. 
	Slodowy showed in \cite[Section 1.5, Theorem 1]{slodowy1980four} that $\pi: S \hookrightarrow \g \twoheadrightarrow \g \gitquo G \cong \cartanh/W$ is the universal graded Poisson deformation of $\CC[X]$, where $S$ is the \textit{Slodowy slice} to a subregular nilpotent orbit of $\g$.  
	
	Let $\mu: T^*R(Q,\delta,0) \to \g$ be the moment map. We have an identification $\cartanh = \{\lambda\in \mathfrak{p}|\lambda\cdot \delta = 0\}$. By remarks at the end of \cite[Section 8]{crawley1998noncommutative}, the map $\mu: \mu\inverse(\cartanh) \to \cartanh$ is obtained from $\pi: S \to \cartanh/W$ by pulling back through the quotient $\cartanh \to \cartanh/W$. Therefore, by \Cref{Corollary of CBH0.4 on smoothness}, we have
	\begin{corollary}{\label{Kleinian singularity codim 2 leaf iff root hyperplane}}
		Suppose $\lambda\in\cartanh$, and $\bar{\lambda}$ is its image in $\cartanh/W$. Then $\pi \inverse (\bar{\lambda})$ is not smooth (i.e. has a codimension 2 leaf) if and only if $\lambda \cdot \alpha = 0$ for some Dynkin root of $Q$. \qed 
	\end{corollary}
	

	\subsection{General conical symplectic singularities}
	In this subsection we describe the conditions on $\lambda \in\cartanh_X$ for $\mc{X}_\lambda =\pi\inverse(\bar{\lambda})$ to have a codimension 2 symplectic leaf. 
	
	Let $\mc{X}_{\CC\lambda} = \pi \inverse (\CC\bar{\lambda})$. We have a graded Poisson morphism $\pi|_{\mc{X}_{\CC\lambda}}: \mc{X}_{\CC\lambda} \to \CC\lambda = \Spec \CC[t]$. Therefore, each fiber of $\pi|_{\mc{X}_{\CC\lambda}}$ is a Poisson subvariety. 
	Let $Z \subset \mc{X}_\lambda$ be a closed Poisson subvariety of dimension $d$. Let $Y = \overline{\CC^* \cdot Z}$ and consider $Y \cap \mc{X}_0$. The following result is standard. 
	
	\begin{proposition}
		$Y \cap \mc{X}_0$ is a $d$-dimensional Poisson subvariety of $X$. \qed 
	\end{proposition}
	
	\begin{corollary}{\label{contract to union of codimension 2 leaf}}
		If $Z \subset \mc{X}_\lambda$ is a codimension $2$ symplectic leaf, then $Y \cap \mc{X}_0$ is the union of the closures of some codimension $2$ symplectic leaves of $X$. 
		\qed 
	\end{corollary}
	\begin{proposition}{\label{codimension 2 leaf deformation has root hyperplane parameters}} 
		
		If $\lambda \in \cartanh_X$, then $\mc{X}_\lambda$ has a codimension 2 symplectic leaf if and only if $\lambda$ lies in a root hyperplane of the Namikawa-Weyl group $W$. 
	\end{proposition}
	\begin{proof}
		We may assume $X$ has at least one codimension 2 symplectic leaf. Otherwise, the proposition is vacuous by \Cref{Ivan's decomposition of cartan space}. 
		
		Let $\mc{L}_i$ be a codimension 2 symplectic leaf of $X$ and $\Sigma_i$ be the slice Kleinian singularity at some $x\in\mc{L}_i$. 
		Let $\lambda_i$ be the projection of $\lambda$ to the Cartan space of $\Sigma_i$ in the decomposition in \Cref{Ivan's decomposition of cartan space}, viewed as an element of $\hat{\cartanh}_i$. 
		Let $\bar{\lambda}_i \in \hat{\cartanh}_i/\widehat{W}_i$ be its image in the quotient. 
		Let $\pi_i: S_i \to \hat{\cartanh}_i/\widehat{W}_i$ be the universal graded Poisson deformation of the Kleinian singularities $\Sigma_i$, and let $ S_{\CC\lambda_i} = \pi_i\inverse(\CC \bar{\lambda}_i) \subset S_i$.
		
		Let $\mf{m}_x$ be the maximal ideal of $\CC[\mc{X}_{\CC\lambda}]$ corresponding to $x$, and define \[(\mc{X}_{\CC\lambda})^{\wedge_x} = \Spec ((\mc{O}_{\mc{X}_{\CC\lambda},x})^{\wedge_{\mf{m}_x}}).\]
		Let $V$ be the tangent space $T_x\mc{L}_i$, and consider $V\times S_{\CC\lambda_i}$, viewed as a scheme over $\CC\lambda_i$. Let $\mf{m}_0$ be the maximal ideal corresponding to $0 \in V\times S_{\CC\lambda_i}$. Let \[(V\times S_{\CC\lambda_i})^{\wedge_0} = \Spec (\CC[V\times S_{\CC\lambda_i}]^{\wedge_{\mf{m}_0}}).\]
		Then $(\mc{X}_{\CC\lambda})^{\wedge_x} \cong (V\times S_{\CC\lambda_i})^{\wedge_0}$ (see \cite[Section 2.3]{losev2022deformations}). 
		
		Consider the completion of the sheaf of relative differentials $(\Omega_{\mc{X}_{\CC\lambda}/\CC\lambda})^{\wedge_x}$ (it is different from $\Omega_{(\mc{X}_{\CC\lambda})^{\wedge_x}/\CC\lambda}$), viewed as a sheaf on $(\mc{X}_{\CC\lambda})^{\wedge_x}$. 
		It only depends on the completed algebra $(\mc{O}_{\mc{X}_{\CC\lambda},x})^{\wedge_{\mf{m}_x}}$. 
		Similarly, we define the completion of the module of relative differentials $(\Omega_{(V\times S_{\CC\lambda_i})/\CC\lambda})^{\wedge_0}$, which only depends on $(V\times S_{\CC\lambda_i})^{\wedge_0}$. Therefore, we have
		\[(\Omega_{(V\times S_{\CC\lambda_i})/\CC\lambda})^{\wedge_0} \cong (\Omega_{\mc{X}_{\CC\lambda}/\CC\lambda})^{\wedge_x}.\tag{$\star$}\]
		
		By \Cref{contract to union of codimension 2 leaf}, $\mc{X}_\lambda$ has a codimension 2 symplectic leaf $Z$ if and only $\pi_\lambda:\mc{X}_{\CC\lambda} \to \CC\lambda $ is non-smooth along a codimension 2 closed Poisson subvariety (i.e. $\overline{\CC^* \cdot Z}$), which intersects $X$ at a union of closure of codimension 2 leaves.
		Equivalently, for some codimension 2 leaf $\mc{L}_i$ of $X$, there is a point $x\in \mc{L}_i$, and a (non-closed) point $\mf{p} $ of $\mc{X}_{\CC\lambda}$ whose closure contains $x$, such that $(\Omega_{\mc{X}_{\CC\lambda}/\CC\lambda})^{\wedge_x}$ has rank larger than $\dim X$ at $\mf{p}$ (viewed as a point in $(\mc{X}_{\CC\lambda})^{\wedge_x}$). 
		By ($\star$), this is equivalent to that $(\Omega_{(V\times S_{\CC\lambda_i})/\CC\lambda})^{\wedge_0}$ has rank larger than $\dim X$ at some non-closed point $\mf{q} \in V\times S_{\CC\lambda_i}$ whose closure contains $0$. This happens if and only if
		$V\times S_{\CC\lambda_i}$ has a codimension 2 symplectic leaf, which is further equivalent to  $\lambda_i$ lying on a root hyperplane of the Cartan space of $S_i$, by \Cref{Kleinian singularity codim 2 leaf iff root hyperplane}. Since $\pi_1({\mc{L}_i})$ act by diagram automorphisms, by restricting to $\cartanh_i = (\hat{\cartanh}_i^*)^{\pi_1({\mc{L}_i})}$, we see $\lambda$ lies in a root hyperplane of $W_i$. 
		
	\end{proof}
	
	\section{Namikawa-Weyl group via codimension 2 roots}{\label{Application to quiver varieties}}
	In this section, we continue to make the assumptions of \Cref{best convention on tilde v}, unless otherwise stated. 
	\subsection{Linear relations between codimension 2 roots}
	We combine the results of \Cref{Codimension 2 leaves in deformed quiver varieties} and \Cref{Codimension 2 leaves in deformed conical symplectic singularity} to describe the Namikawa-Weyl group of the affinization of a smooth quiver variety. 
	Recall the map $\kappa: \mathfrak{p} \to H^2(\M^\theta_0(v,w),\CC ) $ defined in \Cref{subsection taut line bundles}.
	
	Write $X = \M^0_0(v,w)$. For any positive root $v^1$ of $Q$, let $H_{v^1}\subset  \mathfrak{p}$ be the hyperplane $\{\lambda| \lambda\cdot v^1 = 0\}$.
	\begin{proposition}{\label{kappa induce bijection between hyperplanes}}
		The map $\kappa$ gives a bijection between the set \[\{H_{v^1}|v^1 \text{ is a codimension 2 root of Q}\}\] and the set of root hyperplanes of $\cartanh_X$. 
	\end{proposition}
	\begin{proof}
		Since a generic element in $H_{v^1}$ produces a deformation of $\M_0(v,w)$ which contains a codimension 2 leaf, $\kappa(H_{v^1})$ must be a proper subspace of $\cartanh_X$. Since $\kappa$ is surjective by \Cref{k is surjective}, $\kappa(H_{v^1})$ is in fact a hyperplane. 
		By \Cref{codimension 2 leaf deformation has root hyperplane parameters}, $\kappa(H_{v^1})$ is a root hyperplane of $\cartanh_X$. 
		
		Conversely, for any root hyperplane $H$ of $\cartanh_X$, the preimage $\kappa\inverse(H)$ is a hyperplane of $\mathfrak{p}$ since $\kappa$ is surjective. The fiber over any $h\in H$ has a codimension 2 symplectic leaf, so by \Cref{theorem cod 2 leaves in deformed quiver varietyes}, $\kappa\inverse(H)$ must be of the form $H_{v^1}$. It is clear that $\kappa$ and $\kappa\inverse$ give mutually inverse bijections on the set of hyperplanes as desired. 
	\end{proof}
	
	Recall the simply laced Cartan space $\hat{\cartanh}_i^* \cong \CC^{\underline{Q}_0}$ where $\underline{Q}$ is the slice quiver associated to the leaf $\L_i$; it is a finite type quiver by \Cref{BS isotropic decomposition}.
	Note there is a unique Dynkin diagram automorphism $\sigma$ such that $\cartanh_i = (\hat{\cartanh}_i^*)^\sigma$. 
	If $\alpha$ is a root of $\cartanh_i$, then there exists a (unique up to $\sigma$ translation) positive root $\underline{\alpha}$ in $\hat{\cartanh}_i^*$, such that $\alpha = \sum_{j=1}^{m(\alpha)} \sigma^j \underline{\alpha}$, where $m(\alpha)$ is the size of the $\sigma$-orbit of $\underline{\alpha}$. 
	View $\underline{\alpha}$ as an element of $\ZZ_{\ge 0}^{\underline{Q}_0}$. 
	\begin{lemma}\label{lemma identify degree of line bundle}
		Suppose the hyperplane of $\mathfrak{p}$ cut out by $v^1$ maps to the root hyperplane of $\alpha$ in $\cartanh_i$ under the map $\kappa$, and let $\underline{\alpha}$ be defined as above. Then for any $\chi \in \ZZ^{Q_0}$, $\chi\cdot v^1 = \kappa_i(\chi) \cdot \underline{\alpha}$. 
	\end{lemma}
	\begin{proof}
		Let $\mc{L}_i$ be the codimension 2 leaf of $\M_0(v,w)$ corresponding to $\cartanh_i$, and $x\in\mc{L}_i$. 
		Let $S$ denote the slice Kleinian singularity. Let $\lambda$ be a subgeneric parameter annihilated by $v^1$. 
		Let $\lambda_i$ be the image of $\lambda$ under the natural restriction $\mathfrak{p}\to \underline{\mathfrak{p}}$. We know by \Cref{BS isotropic decomposition} that $\underline{Q}^\infty$  is an affine type quiver and the dimension vector is the minimal imaginary root $\delta$. 
			For each $c\in\CC^*$, $\underline{\M}^0_{c\lambda_i}(\delta)$ has $m({\alpha})$ codimension 2 leaves, corresponding to
			the representation types $(\delta-\sigma^j(\underline{\alpha}),1; \sigma^j(\underline{\alpha}),1)$ 
			for $1\le j\le m(\alpha)$.
			
			Pick analytic open neighbourhoods $U$ of $x\in \M_{\CC\lambda}^0(v,w)$, $\underline{U}$ of $0 \in \underline{\M}_{\CC\lambda_i}^0(\delta) \times R_0$ and an isomorphism $\phi:U\to \underline{U}$ that satisfy \Cref{analytic local structure theorem}. Pick $y\in U$ that lies in the unique codimension 2 leaf of $\M_{\varepsilon \lambda}^0(v,w)$ for some small $\varepsilon\neq 0$. Then,  $y':= \phi(y)$ lies in one of the $m(\alpha)$ codimension 2 leaves of $\underline{\M}_{\varepsilon \lambda_i}^0(\delta) \times R_0$. 
			By \Cref{lemma degree of line bundle on slice P1}, for $\theta>0$, the restriction of the line bundle $\sheafO(\chi,\M_\mathfrak{p}^\theta(v,w))$ to $\rho\inverse(y)$ has degree $\chi\cdot v^1$. 
			Thanks to \Cref{remark how to project to cartanh i} and \Cref{kirwan map coincides with chern class}, $\phi$ identifies the vector bundles
			\[\sheafO(\chi,\M_\mathfrak{p}^\theta(v,w))|_U \cong \sheafO(\kappa_i(\chi),\underline{\M}_{\underline{\mathfrak{p}}}^\theta(\delta)) \boxtimes \CC[R_0]|_{\underline{U}}.\] 
			
			\Cref{lemma degree of line bundle on slice P1} then shows the restriction to $(\underline{\rho}\times\id) \inverse(y') \cong \PP^1 \times R_0$ has degree $\underline{\alpha} \cdot \kappa_i(\chi)$. Therefore, $\chi\cdot v^1 = \kappa_i(\chi) \cdot \underline{\alpha}$.
		\end{proof}
		The following corollary enables us to recover the linear relations of roots of $\cartanh_i$ from linear relations of codimension 2 roots of $Q$. 
		\begin{corollary}	\label{corollary relation of cod 2 roots and relation of Weyl group roots}
			Suppose $v^j$ are codimension 2 roots of $Q$ corresponding to the root $\alpha_j$ of $\cartanh_i$, for $j = 1,2,3$. Then the following conditions are equivalent:
			\begin{enumerate}[(i)]
				\item $\alpha_1=\alpha_2+\alpha_3$ in $\cartanh_i$;
				\item  $m(\alpha_1)v^1 = m(\alpha_2)v^2 +m(\alpha_3)v^3 $ in $\ZZ^{Q_0}_{\ge 0}$.
			\end{enumerate}
		\end{corollary}
		\begin{proof}
			We have
			\begin{align*}
				&\alpha_1 = \alpha_2 + \alpha_3\\
				 &\iff \sum_{j=1}^{m(\alpha_1)} \sigma^j \underline{\alpha_1} = \sum_{j=1}^{m(\alpha_2)} \sigma^j \underline{\alpha_2}+ \sum_{j=1}^{m(\alpha_3)} \sigma^j \underline{\alpha_3} \\
				&\iff \sum_{j=1}^{m(\alpha_1)} \sigma^j \underline{\alpha_1} \cdot \kappa_i(\chi) =
				\left( \sum_{j=1}^{m(\alpha_2)} \sigma^j \underline{\alpha_2} + \sum_{j=1}^{m(\alpha_3)} \sigma^j \underline{\alpha_3}\right) \cdot \kappa_i(\chi), \ \forall \chi \in \ZZ^{Q_0} \\
				&\iff m(\alpha_1)\underline{\alpha_1}\cdot \kappa_i(\chi) =( m(\alpha_2)\underline{\alpha_2}  + m(\alpha_3)\underline{\alpha_3})\cdot \kappa_i(\chi),\ \forall\chi \in \ZZ^{Q_0} \\
				&\iff m(\alpha_1)v^1 = m(\alpha_2)v^2 + m(\alpha_3)v^3.
			\end{align*}
			The first equivalence above holds by the definition of $\underline{\alpha}$; the second equivalence holds since $\kappa$ is surjective; the third equivalence holds since $\kappa_i(\chi)$ is invariant under $\sigma$-action; the last equivalence holds by \Cref{lemma identify degree of line bundle}. 
		\end{proof}
		
		\begin{remarklabeled}{\label{Main result}}
			It is classical that if $\cartanh_i$ is type $B,C$ or $F_4$ (we have seen the latter two cases are impossible) then $m(\alpha)=2$ when $\alpha$ is a long root, and $m(\alpha)=1$ when $\alpha$ is a short root. If $\cartanh_i$ is type $G_2$ then $m(\alpha)=3$ when $\alpha$ is a long root, and $m(\alpha)=1$ when $\alpha$ is a short root. If $\cartanh_i$ is simply-laced then all $m(\alpha) = 1$. 
			
			Therefore, thanks to \Cref{corollary relation of cod 2 roots and relation of Weyl group roots}, by examining the root strings of multiply-laced Dynkin diagrams, we see that only the following relations for codimension 2 roots $v^j$, $j\in\{1,2,3\}$, are possible (up to reordering):
			\begin{enumerate}[(a)]
				\item $v^1+v^2=v^3$ 
				\item $v^1+2v^2=v^3$ 
				\item $v^1+v^2=2v^3$ 
				\item $v^1+3v^2=v^3$ 
				\item $v^1+v^2=3v^3$ 
			\end{enumerate}
			In the above list, if the coefficient before $v^j$ is greater than 1, then $v^j$ corresponds to a long root, and the other two roots in the relation are short roots.  
			
			Therefore, we can recover the Namikawa-Weyl group of $\M_0(v,w)$ by finding all linear relations of codimension 2 roots of $Q$ of the form (a)-(e). 
		\end{remarklabeled}
		We say a codimension 2 root $v^1$ is \textit{simple} if it does not satisfy any relation $m_1 v^1 = m_2 v^2 +m_3 v^3 $ for codimension 2 roots $v^2,v^3$ and $m_i\in\{1,2,3\}$. The following corollary, 
		which follows directly from \Cref{corollary relation of cod 2 roots and relation of Weyl group roots}, relates the isotropic decomposition \Cref{BS isotropic decomposition} and codimension 2 roots. It enables us to find all the codimension 2 leaves of $\M_0(v,w)$. 
		\begin{corollary}\label{indecomposable cod 2 roots give iso decomposition}
			If a root appears in an isotropic decomposition, then it is a multiple of an simple codimension 2 root. 
			Conversely, if $v^1$ is an simple codimension 2 root, then a multiple of $v^1$ appears in an isotropic decomposition. 
		\end{corollary}
		
		The following is a consequence of \Cref{kappa induce bijection between hyperplanes} and \Cref{indecomposable cod 2 roots give iso decomposition}. 
		\begin{corollary}\label{different isotropic decomp has different roots}
			If $v^1$ is an simple codimension 2 root, then there is a \textit{unique} isotropic decomposition in which a multiple of $v^1$ appears. In particular, all the roots appearing in the representation types of different codimension 2 leaves are distinct. 
		\end{corollary}

		Let us recall the four types of codimension 2 roots of $Q$, given by \Cref{theorem cod 2 leaves in deformed quiver varietyes}:
		\begin{enumerate}
			\item $v^1$ is a real root,  $\<\nu, v^1\> = 0$, and $\tilde{v} - v^1$ is a root of ${Q^\infty}$.
			\item $v^1$ is an isotropic imaginary root, $\<\nu ,v^1\>=2$, and $ \tilde{v} - v^1$ is a root of ${Q^\infty}$.  
			\item $v^1$ is a non-isotropic imaginary root, $\tilde{v}-n v^1$ is a root for $n=1$ or $2$,  $(\tilde{v} - nv^1, nv^1) = -2$, and $\tilde{v}-mv^1$ is not a root for any $m> n$.
		\end{enumerate}
		
		The relations in \Cref{Main result} are not arbitrary. For $i,j,k\in\{1,2,3\}$, if $\alpha,\beta,\gamma $ are type $(i),(j),(k)$ codimension 2 roots respectively, and there exists $m,n,p\in\{1,2,3\}$ such that $m\alpha + n \beta = p\gamma$, then we write $(i)+(j) = (k)$; otherwise we write $(i)+(j) \neq (k)$.
		
		\begin{theorem}{\label{all possible lin comb}}
			The following (symmetric) table lists all possibilities for $(i)+(j) = (k)$. 
			
			\begin{center}
				\begin{tabular}{ |c|c|c|c| }
					\hline
					+   &   (1)   & (2) & (3)  \\ \hline
					(1) &   (1)   &     &      \\ \hline
					(2) &  (2),(3)    &  (3)  &      \\ \hline
					(3) & (2),(3) &  (3)  & (3)      \\ \hline
				\end{tabular}
			\end{center}
		\end{theorem}

		\begin{proof}[Proof of \Cref{all possible lin comb}]
			We first prove all statements $(i) + (j) \neq (k)$, then give an example for each valid $(i)+(j) = (k)$. 
			
			\begin{enumerate}[a)]
				\item $(1)+(1)\neq (i)$, $(i)+(j)\neq (1)$, $(1)+(i)\neq (1)$ for $i,j\in \{2,3\}$. 
				
				If $\alpha$ is a type (2) or (3) codimension 2 root then $(\tilde{v},\alpha) <0$, but if $\alpha$ is type (1)  then $(\tilde{v},\alpha) =0$.

				\item $(2)+(2)\neq (2)$. 
				
				Suppose $\alpha,\beta,\gamma$ are linearly independent imaginary roots and $m,n,p\in\{1,2,3\}$ such that $m\alpha+n\beta = p\gamma$ as in \Cref{Main result}. Take the Tit's form of both sides, we see $(\alpha,\beta) = 0$. 
				Let $\sigma \in W_Q$ be such that $\sigma \alpha = \delta$ for some affine Dynkin diagram, then $(\sigma \alpha,\sigma\beta) = (\delta,\sigma\beta) = 0$.
				Since $\Supp(m\sigma\alpha+n\sigma\beta) = \Supp(p\sigma\gamma)$ is connected, we must have  $\Supp(\sigma\beta )\subset \Supp(\delta)$.
				By \Cref{lemma imaginary root}, $\sigma \beta$ is a positive isotropic imaginary root, so we must have $\sigma\beta = \delta$, which implies $\alpha=\beta$, contradiction. 
				
				%
				%
				
				%
				
				\item $(3)+(2)\neq (2)$ and $(3)+(3) \neq 2$.  
				
				In fact, if $\alpha$ is a positive nonisotropic imaginary root, $\beta$ is a positive imaginary root, and $m,n >0$ such that $m\alpha + n\beta$ is an isotropic imaginary root, then $(m\alpha + n\beta, m\alpha + n\beta) = 0$, so $(\alpha,\beta)>0$. But if $\sigma\in W_Q$ is such that $(\sigma\alpha,\alpha_i)\le 0$ for all simple roots $\alpha_i$, we would have $(\sigma\alpha,\sigma\beta) > 0$, implying $\sigma\beta <0$. This contradicts that $\beta$ is an imaginary root. 
			\end{enumerate}
			We now provide examples for all $(i)+(j) = (k)$ in the table. 
			\begin{enumerate}[a)]
				\setcounter{enumi}{4}
				\item $(1)+(1) = (1)$, see \Cref{MN weyl group dynkin quiver}. 
				\item $(1) + (2) = (2)$. 
				Consider the following quiver 
				\begin{center}
					\begin{tikzpicture}[scale=0.9]
						\filldraw[black] (0,0) circle (1.5pt);
						\foreach \t in {1,-1}{\filldraw[black] (\t,\t) circle (1.5pt);
							\filldraw[black] (-\t,\t) circle (1.5pt);
							\node at (\t,2.2) [rectangle,draw]  {$1$};
							\draw[-,thick] (0.1*\t,0.1*\t)--(0.9*\t,0.9*\t);
							\draw[-,thick] (0.1*\t,-0.1*\t)--(0.9*\t,-0.9*\t);
							\draw[-,thick] (\t,1.1)--(\t,1.9);}	
						\draw[-,thick] (-.9,-1)--(0.9,-1);
						\node at (-0.4,0) {$\alpha_0$};	
						\node at (-1.4,-1) {$\alpha_1$};
						\node at (1.4,-1) {$\alpha_2$};
						\node at (-1.4,1) {$\beta_1$};
						\node at (1.4,1) {$\beta_2$};
					\end{tikzpicture}
				\end{center}
				We take $v = \alpha_0+\alpha_1+\alpha_2+\beta_1+\beta_2$, and $w$ be $1$ over the vertices (corresponding to the simple roots) $\beta_1,\beta_2$, and 0 elsewhere, as is indicated by the squares. Let $\delta = \alpha_0+\alpha_1+\alpha_2$ which is easily checked to be a root of type (2). The simple real roots $\beta_1$ and $\beta_2$ are both of type (1). The root $\delta+ \beta_1$ is of type (2). 
				This gives examples of $(1)+(2) = (2)$. 
				
				\item $(1)+(2) = (3), (1)+(3) = (2)$. \label{B2 example}
				
				Consider the quiver in part (4) of \Cref{examples of theorem 3.1}.
				In the notations there, $\alpha$ is a type (1) codimension 2 root, $\beta$ and $\alpha+\beta$ are type (2), while $\alpha+2\beta$ is type (3). 
				The relations
				\[\alpha + 2\beta = \alpha+2\beta, \alpha + (\alpha+ 2\beta) =2(\alpha+\beta) \]
				giving examples of   $(1)+(2) = (3)$ and $(1)+(3) = (2)$ respectively. 
				
				\item  $(1)+(3) = (3),(2)+(2) = (3)$ and $(3)+(3) = (3)$. \label{G2 example in proof} 
				
				Consider the quiver and the dimension vector in \Cref{Nakajima G2 quiver via isotropic decomposition}, i.e. 
				$v = 2\alpha_1 +3\beta$, $w$ is 1 over $\alpha_1$ and 0 elsewhere. We see $\alpha_1$ is a type (1) codimension 2 root, $\alpha_1 + 2\beta$, $\alpha_1 + 3\beta$ and $2\alpha_1 + 3\beta$ are type (3) roots, $\beta$ and $\alpha_1 + \beta$ are type (2). 
				We have relations
				\begin{align*}
					& (\alpha_1) + (\alpha_1+ 3\beta) = 2\alpha_1 + 3 \beta;\\
					&(\beta) + (\alpha_1 + \beta) = \alpha_1 + 2\beta;\\
					&(\alpha_1+3\beta) + (2\alpha_1 + 3\beta) = 3 (\alpha_1 + 2\beta)
				\end{align*}
				which are examples of $(1)+(3)=(3), (2)+(2) = (3)$ and $(3)+(3) = (3)$. 
				
				\item  $(2)+(3) = (3)$.  
				Consider the following quiver. 
				\begin{center}
					
					\begin{tikzpicture}[scale=1]
						\foreach \t in {1,-1}
						{
							\filldraw[black] (\t,0) circle (1.5pt);
							\filldraw[black] (2*\t,0) circle (1.5pt);
							\filldraw[black] (3*\t,-1) circle (1.5pt);
							\filldraw[black] (3*\t,1) circle (1.5pt);
							\node at (\t,1.3) [rectangle,draw]  {1};
							\draw[-,thick] (\t*1.1,0)--(\t*1.9,0);
							\draw[-,thick] (\t*3,-0.9)--(\t*3,0.9);
							\draw[-,thick] (\t*2.9,-1)--(\t*2.1,-0.1);
							\draw[-,thick] (\t*2.9,1)--(\t*2.1,0.1);
							\draw[-,thick] (\t,.1)--(\t,1);
						}
						\draw[-,thick] (0.9,0)--(-0.9,0);
						\node at (-1,-0.3) {$\alpha_3$};
						\node at (-2,-0.3) {$\alpha_0$};
						\node at (-3.3,-1) {$\alpha_1$};
						\node at (-3.3,1) {$\alpha_2$};	
						\node at (1,-0.3) {$\beta_3$};
						\node at (2,-0.3) {$\beta_0$};
						\node at (3.3,-1) {$\beta_1$};
						\node at (3.3,1) {$\beta_2$};		
					\end{tikzpicture}
					
				\end{center}
				Write $\delta_1 = \alpha_0+\alpha_1+\alpha_2$ and $\delta_2 = \beta_0 + \beta_1+\beta_2$. Let $v = \delta_1  + \alpha_3 + \beta_3+2\delta_2$, and $w$ be 1 over $\alpha_3,\beta_3$, 0 elsewhere. Let $\alpha = \delta_1 + \alpha_3$ so that $\alpha$ is an isotropic imaginary root, and $(\tilde{v},\alpha) = -2$. Since $\tilde{v} - \alpha = \alpha_\infty + \beta_3 + 2\delta_2$ is an imaginary root, we see $\alpha$ is a type (2) root. 
				Let $\beta = 2\delta_2 + \beta_3$ so that $\beta$ is a non-isotropic imaginary root, and $(\tilde{v}-\beta,\beta) = (\alpha_\infty + \alpha_3 + \delta_1, 2\delta_2 + \beta_3) = -2$. Also, $\tilde{v}-\beta = \alpha_\infty + \delta_1+\alpha_3$ is an imaginary root, and $\tilde{v} \not\ge n\beta$ for $n\ge 2$. Therefore, $\beta$ is a type (3) root. 
				
				Note $v$ is a non-isotropic imaginary root. $(\tilde{v}-\gamma,\gamma) = (\alpha_\infty,\gamma) = -2$, $\tilde{v}-v = \alpha_\infty$ is a root, and clearly $\tilde{v}-nv$ is not a root for $n\ge 2$. Therefore, $v$ is a type (3) root and $v = \alpha + \beta$ gives an example of $(2)+(3)=(3)$. 
			\end{enumerate}
		\end{proof}
		
		\begin{remarklabeled}\label{ogrady example kappa not surjective}
			We finish this subsection with a remark on generalizing our method to quiver varieties $\M_0^0(v)$ without framing. 
			An important ingredient of our approaches is the surjectivity of $\kappa$, \Cref{k is surjective}. For non-framed quiver variety, the source of $\kappa$ is restricted to $\{\lambda\in\mathfrak{p}|\lambda\cdot v = 0 \}$. However, it is not known to our knowledge whether a similar results hold when all components of the dimension vector $v$ are larger than 1, even if it is indivisible, which means $v$ is not a nontrivial integral multiple of a smaller root. For more details, see  \cite[Lemma 4.11]{mcgerty2018kirwan}. 
			
			If $v$ is not indivisible, then $\kappa$ need not be surjective. In fact, let $Q$ be the quiver with a single vertex and 2 edge loops, and let $v=2$. Then the set $\{\lambda\in\mathfrak{p}|\lambda\cdot v = 0 \}$ is just $\{0\}$. 
			On the other hand, it is known that (e.g. \cite[Section 5]{bellamy2021symplectic}) $\M_0^0(v) \cong \mathcal{N} \times \CC^4$ where $\mathcal{N} \subset \sp(4)$ is the closure of the nilpotent orbit $\{B\in\sp(4)|B^2 = 0, \rank B = 2\}$. The variety $\mathcal{N}$ has a symplectic resolution given by $T^*(G/P)$ where $P \subset\Sp(4)$ is the stabilizer of a Lagrangian subspace. Therefore, $T^*(G/P)\times \CC^4 \to \M_0^0(v)$ is a symplectic resolution of singularities. The base of universal deformation is $H^2(T^*(G/P)\times \CC^4,\CC) \cong H^2(G/P,\CC) \neq 0$, e.g. by the Bruhat decomposition. Therefore, $\kappa$ is not surjective. 
		\end{remarklabeled}

		\subsection{Examples}
		We now apply the idea of subgeneric deformation to compute some Namikawa-Weyl groups.
		Let us first consider \Cref{Nakajima G2 quiver via isotropic decomposition} again. 
		
		\begin{example}{\label{Nakajima's G2 example}}
			
			Let $Q,v,w$ be as in \Cref{Nakajima G2 quiver via isotropic decomposition}. Then the codimension 2 roots are listed below: $\alpha_1$, type (1); $\beta$ and $\alpha_1+\beta$, type (2); $\alpha_1+2\beta$, $\alpha_1+3\beta$ and $2\alpha_1+3\beta$, type (3).
			The linear relations between them satisfy the $G_2$ Dynkin diagram: 
			\begin{center}
				\begin{tikzpicture}
					\filldraw[black] (0,0) circle (1.5pt);
					\filldraw[black] (2,0) circle (1.5pt);
					\filldraw[black] (1,1.732) circle (1.5pt);
					\filldraw[black] (3,1.732) circle (1.5pt);
					\filldraw[black] (0,3.464) circle (1.5pt);
					\filldraw[black] (-1,1.732) circle (1.5pt);
					\filldraw[black] (-3,1.732) circle (1.5pt);
					\draw[-,thick] (0.1,0)--(1.9,0);
					\draw[-,thick] (0.1,0.1)--(2.9,1.65);
					\draw[-,thick] (0.05,0.1)--(0.95,1.66);
					\draw[-,thick] (0,0.1)--(0,3.35);
					\draw[-,thick] (-0.1,0.1)--(-2.9,1.65);
					\draw[-,thick] (-0.05,0.1)--(-0.95,1.66);
					\node at (0,-0.3) {0};
					\node at (2,-0.3) {$\alpha_1$};
					\node at (1,2) {$2\alpha_1+3\beta$};
					\node at (3,2) {$\alpha_1+\beta$};
					\node at (-3,2) {$\beta$};
					\node at (-1,2) {$\alpha_1+3\beta$};
					\node at (0,3.8) {$\alpha_1+2\beta$};
				\end{tikzpicture}
			\end{center}
			Here, linear relations are normalized according to short/long roots: if a long root $\gamma$ is involved, we replace $\gamma$ by $3\gamma$. 
			For example, $\alpha_1+(2\alpha_1+3\beta) = 3(\alpha_1+\beta)$; the left hand side is the sum of short roots, and the right hand side is a normalized long root. Therefore, $\M_0^0(v)$ has only one codimension 2 root, and its Namikawa-Weyl group is the Weyl group of type $G_2$.  
		\end{example}
		
		\begin{proposition}{\label{multiply laced only for wild quiver}}
			Suppose $Q$ is a Dynkin quiver or an affine quiver. Then the Namikawa-Weyl group of $\M_0(v,w) $ is a direct product of type $A,D$ or $E$ Weyl groups. 
		\end{proposition}
		\begin{proof}
			For Dynkin or affine $Q$, type (3) codimension roots cannot appear. Suppose a codimension 2 root $\alpha$ of type (1) or (2). We show all $\alpha$-root strings have length at most 2. This will imply all the component of the Namikawa-Weyl group are of simply-laced type. 
			
			If not, then there exists $m,n,p,m',n',p' \in \{1,2,3\}$ and roots $\beta,\gamma,\gamma'$ of type (1), (2) such that $m\alpha+n\beta = p\gamma,m'\alpha+n'\gamma = p'\gamma'$. 
			Since $Q$ is simply laced Dynkin or affine, it is impossible that all of $\alpha,\beta,\gamma,\gamma'$ are real. 
			By \Cref{all possible lin comb}, the only possibility is $\alpha$ is type (1) and $\beta,\gamma,\gamma'$ are type (2). 
			Pair both sides of $m\alpha+n\beta= p\gamma$ with $\tilde{v}$, we see $n=p=1$. Then $(m\alpha+\beta,m\alpha+\beta) = 0$ implies $(\alpha,\beta) = -m$. Similarly, $m'\alpha+n'\gamma = p'\gamma'$ implies $(\alpha,\gamma) = -m'$. But $(\alpha,\gamma) = (\alpha,\beta + m\alpha) = m>0$, contradiction. 
		\end{proof}
		
		We recover McGerty and Nevins' results, \Cref{MN weyl group dynkin quiver}, from \Cref{corollary relation of cod 2 roots and relation of Weyl group roots}. We need a combinatorial lemma, whose proof we postpone to the end of this section. 
		
		\begin{lemma}{\label{root combi lemma}}
			Suppose $Q$ is a type $A,D$ or $E$ Dynkin or affine type quiver, the dimension vector $v$ and framing $w$ are such that $\nu = \Lambda_{w} - \sum v_i\alpha_i$ is dominant, and $w\neq 0$. 
			Then for any positive real root $\alpha$ such that $\<\nu,\alpha\> =0$, $\tilde{v} - \alpha$ is a root of ${Q^\infty}$. 
		\end{lemma}
		The above statement does not hold for wild quivers. For example, consider the following quiver with dimension vector and framing indicated by the numbers. 
		\begin{center}
			\begin{tikzpicture}[scale=0.9]
				\filldraw[black] (0,0) circle (1.5pt);
				\filldraw[black] (1,0) circle (1.5pt);
				\filldraw[black] (2,0) circle (1.5pt);
				\node at (2,0) (alpha3) {$\ $};
				\draw[-,thick] (1.1,0)--(1.9,0);
				\draw[-,thick] (2,0.1)--(2,0.9);
				\draw[-,thick] (0.1,0)--(0.9,0);
				\draw [-,thick] (2.1,-0.15)arc(-160:160:0.5 and 0.3);
				\draw [-,thick] (-0.1,-0.15)arc(-20:-340:0.5 and 0.3);
				\node at (0,-0.4) {$1$};
				\node at (1,-0.4) {$1$};
				\node at (2,-0.4) {$1$};
				\node at (1,0.4) {$\alpha$};
				\node at (2,1.3) [rectangle,draw]  {2};
			\end{tikzpicture}
		\end{center}
		The only simple real root $\alpha$, associated to the vertex without an edge loop, satisfies $\<\nu,\alpha\> = 0$, but $\tilde{v} - \alpha$ does not have connected support and cannot be a root. 
		\begin{corollary}[{\cite[Theorem 5.4]{mcgerty2019springer}}]{\label{MN results recovered}}
			Suppose $Q$ is a Dynkin quiver and $v,w$ are dimension vectors and framing such that $\nu$ is dominant. Let $W_\nu$ denote the Weyl group of the sub-root system $\Phi_\nu$. Then the Namikawa-Weyl group of $ \M_0^0 (v,w) $ is $W_\nu$. 
		\end{corollary}
		
		Here, we do not a priori assume \Cref{best convention on tilde v}; instead, we deduce it from our weaker assumptions. 
		\begin{proof}
			Recall by \Cref{for dominant nu and finite or affine quiver affinization = affine} that $\mu$ is flat and $ \M_0^0 (v,w) \cong \M_0 (v,w) $. 	
			Suppose, for a contradiction, that $(\tilde{v},\alpha_\infty) >0$. We have seen in the proof of \Cref{affinization is an aff quiv variety by Maffei} that $(\alpha_\infty,v) = -1$ in this case, and $v$ is a (real) root supported on $Q$. Then $(\tilde{v}, v) = 1$, contradicting the assumption that $\nu$ is dominant.
			Therefore, $(\tilde{v},\alpha_i) \le 0$ for all $i\in Q_0^\infty$, i.e. \Cref{best convention on tilde v} is satisfied. 
			
			Only type (1) codimension 2 roots can appear. Let $\beta$ be a maximal root of $Q$ with $\<\nu,\beta\> = 0$. 
			By \Cref{root combi lemma}, all positive roots $v^1\le \beta$ satisfies conditions of a type (1) codimension 2 root.  
			The linear relations between these codimension 2 roots are exactly the linear relations between the sub root system generated by $\Supp(\beta)$. 
			Let $W_\beta$ denote the Weyl group of this root system. 
			Repeat for all such maximal $\beta$, we see the Weyl group $W = \prod_{\beta\in \Phi_\nu^{\max}} W_\beta$, where $\Phi_\nu^{\max}$ denotes the set of maximal positive roots of $Q$ vanishing at $\nu$. 
			Equivalently, $W = W_\nu$. 
		\end{proof}
		
		We can generalize the above result to affine type quivers. 
		
		Let $Q$ be affine type. By the construction of \Cref{affinization is an aff quiv variety by Maffei} and the canonical decomposition \Cref{CB canonical decomposition}, we see that for any $v,w$, we can reduce the computation of the Namikawa-Weyl group of $\M_0(v,w)$ to that of 
		\begin{enumerate}
			\item $\M^0_0(Q',v',w')$, where $Q'$ is a finite type quiver and $(v')^\sim \in \Sigma_0$, or
			\item $\M_0^0(Q,v',w')$, so that $\tilde{v}\in\Sigma_0$. 
		\end{enumerate} 
		Therefore, we may assume \Cref{best convention on tilde v}. Let $\delta$ be the minimal positive imaginary root of $Q$. Let $W_{Q_{fin}}$ denote the Weyl group of $Q_{fin}$, the Dynkin quiver corresponding to $Q$.
		Let $(\Phi_\nu^{\max})'$ be the set of maximal positive real roots $\beta$ of $Q$ such that $\<\nu,\beta\> = 0$ and $\Supp(\beta)\subsetneq Q_0$. Let $W_\beta$ denote the Weyl group corresponding to ${\Supp(\beta)}$ which is a Dynkin diagram. 
		Write $W$ be the Namikawa-Weyl group of $\M_0^0(v,w)$. 
		
		\begin{corollary}{\label{extend mn result to affine type}}
			The group $W$ takes the following form. 
			\begin{enumerate}
				\item If $\<\nu,\delta \> = 2$ and $\tilde{v} - \delta$ is a root, then $W = \prod_{\beta\in (\Phi_\nu^{\max})'}W_\beta \times \ZZ/2\ZZ$.
				\item Otherwise, $W = \prod_{\beta\in (\Phi_\nu^{\max})'}W_\beta$ 
			\end{enumerate}
			%
			
		\end{corollary}
		\begin{proof}
			
			%
			We first show $(\tilde{v},\delta)\le -2$. In fact, suppose $(\tilde{v},\delta) \ge -1$. Then since $(\alpha_i,\delta) = 0$ for all $i\in Q_0$, we conclude $(\alpha_\infty,v) \ge -1 $. But then $(\alpha_\infty,\tilde{v}) \ge 1$, contradicting \Cref{best convention on tilde v}. 
			
			Only type (1) and (2) codimension 2 roots can appear, and a type (2) root must be $\delta$. 
			Suppose $v^1$ is type (1). Then since $\nu$ is dominant, $\<\nu,\delta\> \ge 2$ and $\<\nu,v^1 \> = 0$, we have $\Supp(v^1) \subsetneq Q_0$, i.e. the support is Dynkin. Let $\beta$ be maximal among such roots. Then by \Cref{root combi lemma}, every real root $\alpha \le \beta$ is a type (1) codimension 2 root. These real roots give sub-root systems supported on $\Supp\beta$ and corresponding Weyl groups $W_{\beta}$. 
			
			Suppose $v^1 = \delta$ is a type 2 codimension 2 root, which implies $\tilde{v} - \delta$ is a root and $(\tilde{v},\delta) = -2$. By \Cref{all possible lin comb}, $\delta$ is not a linear combination of other (real) codimension 2 roots. Therefore, it contributes a $\ZZ/2\ZZ$ component in the Namikawa-Weyl group. 
			%
			%
		\end{proof}
		Let us finish by examining a related problem, i.e. the Namikawa-Weyl group of symplectic quotient singularities, see \cite{bellamy2020birational}. 
		\begin{example}\label{BC 20 quotient singularity result}
			Let $Q$ be affine type, $v = m\delta$ and $w$ be 1 over the extended vertex and $0$ elsewhere.
			Recall $\mu$ is flat by \Cref{for dominant nu and finite or affine quiver affinization = affine}, so $\M_0(v,w) = \M_0^0(v,w)$. 
			It is classical that $\M_0^0(v,w) $ is the symplectic quotient singularity $\CC^{2m}/(\Gamma\rtimes \mathfrak{S}_m)$, where $\Gamma$ is the finite subgroup of $\SL(2,\CC)$ corresponding to $Q$ under McKay correspondence and $\mathfrak{S}_m$ is the symmetric group; see e.g. \cite[Section 1]{bellamy2020birational}.
			
			In this situation, \Cref{best convention on tilde v} is not satisfied. Nevertheless, it is easy to write down the canonical decomposition of $\M_0^0(v,w)$: 
			\[\M_0^0(v,w) = S^m\M_0^0(\delta)\times pt.\]
			Clearly, $\M_0^0(v,w)$ has 2 codimension 2 leaves. Let $\alpha_0,\cdots,\alpha_n$ be the simple real roots of $Q$ and write $\delta = \sum_{i=0}^n c_i\alpha_i$. Then, the two codimension 2 leaves correspond to representation types 
			\[\tau_1 = (\alpha_\infty,1; \alpha_0,c_0;\cdots;\alpha_n,c_n; \delta,m-1)\]
			and 
			\[\tau_2 = (\alpha_\infty,1; \delta,1; \delta,m-1).\]
			The first leaf gives a Namikawa-Weyl group component of type $W_{Q_{fin}}$, and the second leaf gives a $\ZZ/2\ZZ$. Therefore, the Namikawa-Weyl group of $\CC^{2m}/(\Gamma\rtimes \mathfrak{S}_m)$ is $W_{Q_{fin}} \times \ZZ/2\ZZ$. This is exactly 
			\cite[Proposition 2.2]{bellamy2020birational}.
		\end{example}
		
		We now prove \Cref{root combi lemma}.
		\begin{proof}[Proof of \Cref{root combi lemma}]\label{proof of root combi lemma}
			Suppose $\beta$ is a maximal element among 
			\[\{\beta|\<\nu,\beta\> = 0, \beta \text{ is a positive root}\}.
			\]
			First we show that $\beta$ must be a real root. In fact, if $\beta$ is imaginary (for affine type $Q$), then $(\beta,v) = 0$, so $\<\Lambda_w, \beta\> = 0$. But $\Supp(\beta) = Q_0$ and $w\neq 0$, contradiction. Note that if $\beta = \sum_i m_i\alpha_i$ where $\alpha_i$ are simple roots and $m_i\in \ZZ_{>0}$, then $(\tilde{v},\alpha_i) = 0$ since $\nu$ is dominant. Note also that, if $i\in\Supp\beta$, then $(\alpha_i,\beta) \ge 0$; otherwise $s_i(\beta) >\beta$ and $(\tilde{v}, s_i\beta) =0$, contradicting the maximality of $\beta$. 
			
			\textit{Step 1.} We show $\tilde{v}- \beta$ has connected support. 
			
			If not, let $v_1,v_2,...,v_r$ be the restrictions of $\tilde{v}- \beta$ to the connected components of $\Supp(\tilde{v}- \beta)$; that is, $\sum_{i=1}^r v_i = \tilde{v}- \beta$, and $\Supp(\tilde{v}- \beta) =  \bigsqcup _{i=1}^r \Supp(v_i)$. 
			Exactly one of $\Supp(v_i)$, say $\Supp(v_1)$, contains the vertex $\infty$. For $i\ge 2$, $\Supp(v_i)$ is a proper subdiagram of $Q$ (since it is not connected to $\infty$), and is therefore a Dynkin diagram. 
			
			Let $i\in \Supp(v_2)$, $\alpha_i$ be the simple root attached to $i$, and consider $(\tilde{v}-\beta,\alpha_i) = (v_2,\alpha_i)$; the equality holds since $\Supp(v_2)$ is a connected component of $\Supp(\tilde{v}-\beta)$. We have the following possibilities for $i$. 
			
			\begin{enumerate}
				\item $i\in\Supp(\beta)$. 	
				
				In this case $(\tilde{v},\alpha_i)= 0$ and $(v_2,\alpha_i) = (\tilde{v}-\beta,\alpha_i) = -(\beta,\alpha_i) \le 0$ by maximality of $\beta$. 
				
				\item $i\not \in \Supp(\beta)$ and is not connected to $\Supp(\beta)$ by an edge.
				
				Then $(\alpha_i,\beta) = 0$, and $(v_2,\alpha_i) =(\tilde{v}-\beta,\alpha_i) = (\tilde{v},\alpha_i) \le 0$ since $\nu$ is dominant. 
				\item  $i$ is connected to $\Supp(\beta)$ by an edge and $(\alpha_i,\beta) = -1$. 
				
				By maximality of $\beta$, we know that $(\tilde{v},s_i\beta) =(\tilde{v},\beta + \alpha_i) < 0$; therefore $(\tilde{v},\alpha_i)\le -1$. Then $(v_2,\alpha_i) = (\tilde{v}-\beta,\alpha_i)  \le 0$. 
				
				\item $i$ is connected to $\Supp(\beta)$ by an edge, and $(\alpha_i,\beta) \le -2$.
				
				We still have $(\tilde{v},\alpha_i)\le -1$.
				Note that both $\beta$ and $s_i\beta = \beta - (\alpha_i,\beta)\alpha_i$ are roots, so $\beta+\alpha_i$ is also a root. 
				Now consider $(\alpha_i+\beta,\alpha_i+\beta) = 4+2(\alpha_i,\beta)$.  
				Since $Q$ is of Dynkin or affine type, we must have $(\alpha_i,\beta) = -2$. In this case, if $Q$ is affine type, and $\alpha_i+\beta = \delta$, the minimal positive imaginary root. 
			\end{enumerate}
			
			Suppose for all $i \in \Supp(v_2)$, $i$ falls in case (1), (2) or (3) above. Then $(v_2,\alpha_i) \le 0$ for all $i\in \Supp(v_2)$. This contradicts that $\Supp(v_2)$ is a Dynkin diagram. 
			
			Therefore, let $i \in \Supp(v_2)$ fall in case (4) and $(v_2,\alpha_i) = (\tilde{v}-\beta,\alpha_i) >0$. Since $(\tilde{v},\alpha_i)\le -1$ and $(\beta,\alpha_i) = -2$, we must have $(\tilde{v}-\beta,\alpha_i) =(v_2,\alpha_i) =1$. 
			Therefore, for all $i\neq i'\in \Supp(v_2)$, $(v_2,\alpha_{i'}) = (\tilde{v}-\beta,\alpha_{i'}) = (\tilde{v},\alpha_{i'}) - (\beta,\alpha_{i'}) = (\tilde{v},\alpha_{i'}) - (\delta - \alpha_i,\alpha_{i'}) = (\tilde{v},\alpha_{i'}) + (\alpha_i,\alpha_{i'}) \le 0$. 
			
			We conclude that $\Supp(v_2)$ is a proper subdiagram of $Q$ that contains an extended vertex. By checking the position of extended vertices of all affine type quivers, we see $\Supp(v_2)$ must be of type $A_n$, $D_n$, $E_6$ or $E_7$. In terms of \cite[Table 2]{onishchik2012lie} (and compare to \cite[Table 6]{onishchik2012lie}), $i$ may correspond to 
			\begin{enumerate}[a.]
				\item any column of the Cartan matrix if $\Supp(v_2)$ is type $A_n$; 
				\item the first column or the last two columns of the Cartan matrix if $\Supp(v_2)$ is type $D_n$;
				\item the first or the fifth column of the Cartan matrix if $\Supp(v_2)$ is type $E_6$;
				\item the first column of the Cartan matrix if $\Supp(v_2)$ is type $E_7$.
			\end{enumerate}
			By inspection, we see there is no positive integral vector satisfying $(v_2,\alpha_i) = 1, (v_2,\alpha_j) \le 0$ for all $j\neq i$. Therefore, the component $v_2$ of $v-\beta$ cannot exist, and we conclude $\tilde{v}-\beta$ has connect support. 
			
			\textit{Step 2.} 
			We show $\tilde{v}-\beta$ is a root. We have the following cases.
			
			\begin{enumerate}[I.]
				\item For all $i\in\Supp({v}-\beta)$, $(\tilde{v} - \beta, \alpha_i)\le 0$. There are 2 subcases. 
				\begin{enumerate}[i.]
					\item $(\tilde{v} - \beta, \alpha_\infty)\le 0$. In this case, $(\tilde{v} - \beta, \alpha_j)\le 0$ for all $j\in Q^\infty_0$, and $\tilde{v} - \beta$ is an imaginary root. 
					\item $(\tilde{v} - \beta, \alpha_\infty) \ge 1$. In this case, note that $(\tilde{v} - \beta, \alpha_\infty) = (\alpha_\infty +v-\beta ,\alpha_\infty) = 2 + (v-\beta ,\alpha_\infty) \le 1$, so $(\tilde{v} - \beta, \alpha_\infty) =1$. Thus, there is a unique edge connecting the vertex $\infty$ to some other vertex in $\Supp(\tilde{v}-\beta)$; we call this vertex $i_1$. The multiplicity of $\alpha_{i_1}$ in $\tilde{v}-\beta$ is 1. See \Cref{sketch for final lemma} for a sketch; there the numbers below the vertices are the multiplicities of the corresponding simple roots in $\tilde{v}-\beta$. 
					\begin{figure}[h]
						\centering
						\begin{tikzpicture}[scale=1]
							\node at (4.3,0) [circle,draw]  {else};
							\filldraw[black] (3,0) circle (1.5pt);
							\filldraw[black] (2,0) circle (1.5pt);
							\draw[-,thick] (2.1,0)--(2.9,0);
							\draw[-,thick] (3.1,0)--(3.9,0);
							\node at (1,0) {$\tilde{v}-\beta=$};
							\node at (2,0.3) {$\infty$};
							\node at (2,-0.3) {$1$};
							\node at (3,0.3) {${i_1}$};
							\node at (3,-0.3) {$1$};
						\end{tikzpicture} 
						\caption{ }
						\label{sketch for final lemma}
					\end{figure}
					
					Therefore, $(\tilde{v}-\beta,\alpha_{i_1}) \le 1$.
					
					If $(\tilde{v}-\beta,\alpha_{i_1}) \le -1$, then $(s_\infty(\tilde{v}-\beta),\alpha_j) = (\tilde{v}-\beta-\alpha_\infty,\alpha_j) \le 0$ for all $j\in Q_0^\infty$, and we get an imaginary root.
					
					If $(\tilde{v}-\beta,\alpha_{i_1}) =1$ then $\tilde{v}-\beta = \alpha_1+ \alpha_\infty$ is a real root. 
					
					Finally, if $(\tilde{v}-\beta,\alpha_{i_1}) =0$ then there is a unique vertex $i_2\in \Supp(\tilde{v}-\beta)$ such that $i_1$ is connected only to $\infty$ and $i_2$ by a single edge, and the multiplicity of $\alpha_{i_2}$ in $\tilde{v}-\beta$ is 1. In \Cref{sketch for final lemma}, $i_2$ belongs to the ``else" part. Apply the previous arguments to $s_\infty(\tilde{v}-\beta) = v-\beta$ (now $\alpha_{i_1}$ plays the role of $\alpha_\infty$ before). After finitely many iterations we see $s_{i_k}s_{i_{k-1}}\ldots s_\infty(\tilde{v}-\beta)$ is a root for a sequence of vertices $i_k$, and therefore $\tilde{v}-\beta$ is a root. 
				\end{enumerate}
				\item For some $i\in \Supp({v}-\beta)$, $(\tilde{v} - \beta, \alpha_i)\ge 1$. Since $(\tilde{v},\alpha_i) \le 0$, we have $(\alpha_i, \beta) \le -1$. Then the vertex $i$ is connected to $\beta$; so by maximality of $\beta$ we have $(\tilde{v},\alpha_i) \le -1$, so $(\alpha_i, \beta) \le -2$. In other words, $i$ must fall in case (4) in Step 1. In particular, $Q$ is affine type, $i\not\in \Supp(\beta)$, $(\tilde{v},\alpha_i) = -1$ and $\delta = \alpha_i+\beta$.
				Therefore 
				\[ (\tilde{v},\delta) =(\tilde{v},\alpha_i+\beta) =  (\alpha_\infty,\delta)=-1,\]
				so the framing $w$ has the form $w_{i'} =1$ for some $i'$ such that multiplicity of $\alpha_{i'}$ in $\delta$ is $1$, and $w_{j} =0$ for all other $j$. 
				
				We claim $i'=i$. Otherwise, $ i'\in \Supp(\beta)$, and $(v,\alpha_{i'}) = -(\alpha_\infty,\alpha_{i'}) = 1$. 
				Also, $(v,\alpha_i) = (\tilde{v},\alpha_i) = -1$, and $(v,\alpha_j) = 0$ for $Q_0\ni j\neq i,i'$. 
				The set for all such $v = (v_k)|_{k\in Q_0}$, i.e. the integral solution set of the linear system
				\[\left\{ 
				\begin{array}{ccc}
					(v,\alpha_{i'}) &= 1 & \\
					(v,\alpha_i) &= -1 & \\
					(v,\alpha_j) &= 0 & \text{for all $j\in Q_0, j\neq i,i'$}
				\end{array}
				\right.\tag{$\star$}\] 
				is the set of integral elements in the affine line $v'+\CC\delta$, where $v'$ is any special solution. 
				
				To find a special solution $v'=(v'_k)|_{k\in Q_0}$, we may let $v'_i = 0$. Then $\Supp(v') \subset \Supp(\beta)$, a Dynkin diagram, and the problem is reduced to solving the system 
				\[\left\{ 
				\begin{array}{ccc}
					v_i &= 0 &\\
					(v,\alpha_{i'}) &= 1 & \\
					(v,\alpha_j) &= 0 & \text{for all $j\in \Supp(\beta), j\neq i'$}
				\end{array}
				\right.\tag{$\star'$}\] 
				which has a unique solution. 
				Recall that the multiplicity of $\alpha_{i'}$ in $\delta$ is 1. In terms of \cite[Table 2]{onishchik2012lie},  $i'$ may correspond to any column of the Cartan matrix if $Q$ is type $A_n$; 
				$i'$ may correspond to the first or last 2 columns of the Cartan matrix if $Q$ is type $D_n$;
				$i'$ may correspond to the first or the fifth column of the Cartan matrix if $Q$ is type $E_6$;
				$i'$ may correspond to the first column of the Cartan matrix if $Q$ is type $E_7$;
				and $Q$ cannot be type $E_8$. 		
				By inspection, we see $v'$ cannot be integral in any case. 		
				But the multiplicity of $\alpha_i$ in $\delta$ is 1, and therefore an integral solution for $v$ has the form $v'+n\delta$. We conclude there is no integral solution for $v$. 
				
				Therefore $w_i=1, w_j = 0$ for $j\neq 0$. Then for $j\in\Supp(\beta)$, $(v,\alpha_j) = (\tilde{v},\alpha_j) = 0$; and $(v,\alpha_i) = (\tilde{v},\alpha_i) + 1 = 0$. Hence $v = n\delta$ for some positive integer $n$, and it is clear that $\tilde{v}-\beta$ is a real root. 
			\end{enumerate}
			
			This finishes \textit{Step 2} of the proof. 	
			
			\textit{Step 3.} We show that for any positive real root $\alpha\le \beta$, $\tilde{v} - \alpha$ is a root. 
			
			In fact, there is a sequence of simple reflections $s_{i_1},...,s_{i_m}$ where each $i_k\in\Supp(\beta)$ such that $\alpha = s_{i_1}s_{i_2}...s_{i_m}\beta$. Since $(\tilde{v},\alpha_{i_k}) = 0$, we have $s_{i_1}s_{i_2}...s_{i_m}(\tilde{v}-\beta) = \tilde{v} - \alpha$. Hence $\tilde{v} - \alpha$ is also a root. This finishes the proof of the lemma. \qedhere 
		\end{proof}

\bibliographystyle{amsplain}

%
		
	\end{document}